\documentclass[10pt,reqno]{amsart}
\usepackage[colorlinks=true,
pdfstartview=FitV, linkcolor=cyan, citecolor=magenta,
urlcolor=]{hyperref}
\usepackage{enumerate}
\usepackage{mathabx}
\usepackage{amssymb,amsfonts, amsmath}
\usepackage{hyperref}
\newtheorem{theorem}{Theorem}[section]
\newtheorem{lemma}[theorem]{Lemma}

\newtheorem{corollary}[theorem]{Corollary}
\newtheorem{definition}[theorem]{Definition}

\newtheorem{question}[theorem]{Question}

\newtheorem{claim}[theorem]{Claim}
\newtheorem{rmks}[theorem]{\normalfont{\em{Remarks}}}
\newtheorem{rmk}[theorem]{\normalfont{\em{Remark}}}

\usepackage[title]{appendix}

\makeatletter
\renewcommand*\env@matrix[1][\arraystretch]{%
\edef\arraystretch{#1}%
\hskip -\arraycolsep
\let\@ifnextchar\new@ifnextchar
\array{*\c@MaxMatrixCols c}}
\makeatother
\AtEndDocument{\bigskip{\footnotesize%
\textsc{Max Planck Institute for Mathematics in the Sciences, Inselstra{\ss}e 22, 04103 Leipzig, Germany } \par  
 \textit{E-mail address}: \texttt{konstantinos.tsouvalas@mis.mpg.de} }}
\date{\today}
\title[Robust quasi-isometric embeddings inapproximable by Anosov representations]{Robust quasi-isometric embeddings inapproximable by Anosov representations
}
\date{\today}

\author{Konstantinos Tsouvalas}

\begin{document}

\frenchspacing

\maketitle

\begin{abstract} Let $\mathbb{K}=\mathbb{R}$ or $\mathbb{C}$. For all but finitely many $m\in \mathbb{N}$, we exhibit the first examples of non-locally rigid, Zariski dense, robust quasi-isometric embeddings of hyperbolic groups in $\mathsf{SL}_m(\mathbb{K})$ which are not limits of Anosov representations. As a consequence, we show that higher rank analogues of Sullivan's structural stabilty theorem and of the density theorem for Kleinian groups fail for Anosov representations in $\mathsf{SL}_m(\mathbb{C}), m\geq 30$. 
\end{abstract}

\section{Introduction}

Anosov representations of Gromov hyperbolic groups in real semisimple Lie groups were introduced by Labourie in his work on the Hitchin component \cite{Labourie} and further generalized by Guichard--Wienhard in \cite{GW}. Anosov representations have been extensively studied in the past two decades and today they are recognized as the correct higher rank analogues of convex cocompact subgroups of rank one Lie groups. 

Let $G$ be a real semisimple Lie group and $\Gamma$ a hyperbolic group. A representation $\rho:\Gamma \rightarrow G$ is called {\em robustly discrete faithful} (resp. {\em robustly a quasi-isometric embedding}) if every representation, sufficiently close to $\rho$ in the space of representations $\textup{Hom}(\Gamma,G)$, is discrete and faithful (resp. a quasi-isometric embedding). An important property of Anosov representations is that they are {\em structurally stable} quasi-isometric embeddings, i.e. form an open subset of the space of representations in the target Lie group consisting entirely of quasi-isometric embeddings  (see \cite{Labourie}, \cite[Thm. 5.13]{GW}). Sullivan \cite{Sullivan} established that a non-locally rigid discrete subgroup of $\mathsf{PSL}_2(\mathbb{C})$, with the property that all nearby deformations of its inclusion in $\mathsf{PSL}_2(\mathbb{C})$ remain faithful, is necessarily hyperbolic and convex cocompact. In light of the fact that Anosov representations are stable, it is natural to ask whether an analogue of Sullivan's structural stability theorem holds for Anosov representations in higher rank Lie groups. The following question was asked by R. Potrie:

\begin{question}\label{question} \textup{(Potrie \cite[Q.5]{Potrie-ICM}, \cite[Q.1]{Potrie})} Let $\Gamma$ be a (non-elementary) hyperbolic group and $G$ a linear semisimple Lie group of real rank at least $2$.

\noindent \textup{(i)} If a representation $\rho:\Gamma \rightarrow G$ is robustly discrete and faithful, is $\rho$ Anosov with respect to a pair of opposite parabolic subgroups of $G$?\\
\noindent \textup{(ii)} If a representation $\rho:\Gamma \rightarrow G$ is robustly quasi-isometric embedding, is $\rho$ Anosov with respect to a pair of opposite parabolic subgroups of $G$? \end{question}

\noindent Question \ref{question} (i) was also independently asked by F. Kassel in \cite[p. 1165]{Kassel-ICM} with the additional requirement that $\rho$ is non-locally rigid (i.e. $\rho$ is a limit of representations not conjugate to it). A positive answer to Question \ref{question} (ii) was given in \cite{CPL24}, in the case where $G=\mathsf{SL}_3(\mathbb{R})$, $\Gamma$ is a non-elementary free group and $\rho:\Gamma \rightarrow \mathsf{SL}_3(\mathbb{R})$ is a reducible representation. A negative answer to Question \ref{question} (i) was established in \cite[Exmp. 8.5]{DGKLM21} for $G=\mathsf{SL}^{\pm}_5(\mathbb{R})$, by exhibiting an example of a discrete faithful representation $\psi:\Delta \rightarrow \mathsf{SL}^{\pm}_5(\mathbb{R})$, of a hyperbolic Coxeter group $\Delta$, which is not Anosov and all of whose nearby deformations are projective Anosov representations.

Let $\mathbb{K}=\mathbb{R}$ or $\mathbb{C}$. The goal of the present paper is to exhibit the first examples of non-locally rigid, robust quasi-isometric embeddings of hyperbolic groups in $\mathsf{SL}_m(\mathbb{K})$ which are not limits of Anosov representations in $\mathsf{SL}_m(\mathbb{K})$. For $m\in \mathbb{N}$ and $1\leq k \leq m-1$, $\textup{Anosov}_{k}(\Gamma,\mathsf{SL}_m(\mathbb{K}))\subset \textup{Hom}(\Gamma,\mathsf{SL}_m(\mathbb{K}))$ denotes the open subset of $k$-Anosov representations of a hyperbolic group $\Gamma$ in $\mathsf{SL}_m(\mathbb{K})$ (see Definition \ref{Anosov-def}). For $n\geq 2$, let $d(n)\in \mathbb{N}$ be the minimal dimension of an irreducible $1$-proximal representation of the rank one Lie group $\mathsf{Sp}(n,1)$ over $\mathbb{R}$. Our main result is the following.

\begin{theorem}\label{mainthm} Let $n\geq 2$ and $m>\frac{1}{2}d(n)^2$ integers. There exist $\Gamma_1,\Gamma_2,\ldots, \Gamma_{\ell}<\mathsf{Sp}(n,1)$ uniform lattices and  an open subset $\Omega \subset \textup{Hom}(\Gamma_1\ast \cdots \ast \Gamma_{\ell},\mathsf{SL}_{m}(\mathbb{K}))$ with the following properties:

\noindent \textup{(i)} Every representation in $\Omega$ is a quasi-isometric embedding.\\
\noindent \textup{(ii)} $\Omega \cap \textup{Anosov}_k(\Gamma_1\ast \cdots \ast \Gamma_{\ell},\mathsf{SL}_{m}(\mathbb{K}))$ is empty for every $k=1, \ldots, m-1$.\\
\noindent \textup{(iii)} Every representation in $\Omega$ is non-locally rigid and Zariski dense in $\mathsf{SL}_{m}(\mathbb{K})$.\end{theorem}

Theorem \ref{mainthm} settles in the negative Question \ref{question} (i) and (ii) and the question in \cite[p. 1165]{Kassel-ICM}, for $G=\mathsf{SL}_m(\mathbb{K})$ and all but finitely many $m \in \mathbb{N}$. The examples of Theorem \ref{mainthm} are the first known examples of discrete faithful representations of hyperbolic groups in higher rank complex special linear groups which fail to be limits of Anosov representations. As a consequence, our result shows that higher rank analogues of Sullivan's structural stability theorem \cite{Sullivan} and of the density theorem of Kleinian groups, established by a series of papers \cite{BrBr, BCM, NS, Ohshika}, fail for Anosov representations in $\mathsf{SL}_m(\mathbb{C})$ and $m\geq 30$ (see Corollary \ref{cor-1}).
\par The free product $\Gamma_1\ast \cdots \ast \Gamma_{\ell}$ admits Anosov representations in $\mathsf{SL}_m(\mathbb{K})$, thanks to known combination theorems for free products of $1$-Anosov groups \cite{DGK2, DK, DKL} (see also \cite[Thm. 1.3]{DoubaT}). Nevertheless, all such representations avoid sufficiently small neighbourhoods of our examples in the space of representations. Crucial for the proof of Theorem \ref{mainthm} is the fact that complex representations of the uniform lattices $\Gamma_1,\ldots, \Gamma_{\ell}<\mathsf{Sp}(n,1)$, $n\geq 2$, are locally rigid by Corlette's Archimedean superrigidity \cite{Corlette}. Every representation \hbox{$\rho:\Gamma_1 \ast \cdots \ast \Gamma_{\ell}\rightarrow \mathsf{SL}_{m}(\mathbb{K})$} in the open set $\Omega$ is constructed as a ping-pong combination of Anosov representations $\rho_1:\Gamma_1\rightarrow \mathsf{SL}_{m}(\mathbb{K})$ and \hbox{$\rho_i:\Gamma_i\rightarrow \mathsf{SL}_{m}(\mathbb{K})$}, $i=2,\ldots, \ell$, such that the set $I_{\rho,1}=\{k\in \mathbb{N}: \rho_1 \ \textup{is}\  k\textup{-Anosov}\}$ is disjoint from $I_{\rho,i}=\{k\in \mathbb{N}: \rho_{i} \ \textup{is}\  k\textup{-Anosov}\}$. This key property, combined with the local rigidity of $\rho_1$ and $\rho_i$, guarantees that $\rho$ is not a limit of Anosov representations of $\Gamma_1 \ast \cdots \ast \Gamma_{\ell}$ in $\mathsf{SL}_m(\mathbb{K})$. Finally, we note that every representation $\rho \in \Omega$ is Zariski dense and non-locally rigid, since it is possible to conjugate the free factors, with generic elements of $\mathsf{GL}_{m}(\mathbb{K})$, to produce non-trivial deformations arbitrarily close to $\rho$.
\medskip

\noindent {\bf Acknowledgements.} I am grateful to Le\'on Carvajales for motivating discussions on Anosov representations. I would also like to thank Rafael Potrie and Fanny Kassel for their questions and Richard Canary and Anna Wienhard for their comments. Moreover, I would like to thank the anonymous referees for carefully reading the paper and their suggestions that improved the exposition as well as for pointing me to \cite[Prop. 8.2]{A'Campo-Burger} that improved (iii) in Theorem \ref{mainthm} from the previous version.

\section{Background}\label{background}
Let $\mathbb{K}=\mathbb{R}$ or $\mathbb{C}$. For $d \geqslant 2$, we equip $\mathbb{K}^d$ with the standard inner product $\langle \cdot, \cdot\rangle$ and its canonical basis $(e_1,\ldots,e_d)$. Let also $\mathsf{Mat}_{d\times d}(\mathbb{K})$ be the algebra of $d\times d$ matrices with entries in $\mathbb{K}$, equipped with the standard operator norm $||\cdot||$  defined as $||A||:=\sup\{||Av||: ||v||=1\}$.

For a matrix $g\in \mathsf{GL}_d(\mathbb{K})$ denote by $\sigma_1(g)\geq \cdots \geq \sigma_d(g)$ (resp. $\lambda_1(g)\geq \cdots \geq \lambda_d(g)$) the singular values (resp. moduli of eigenvalues) of $g$ in non-increasing order. Denote by $\mathsf{K}_d$ the maximal compact subgroup of $\mathsf{GL}_d(\mathbb{K})$, so that, after a suitable conjugation, $\mathsf{K}_d=\mathsf{O}(d)$ if $\mathbb{K}=\mathbb{R}$ and $\mathsf{K}_d=\mathsf{U}(d)$ if $\mathbb{K}=\mathbb{C}$. The standard Cartan decomposition we consider is $$\mathsf{GL}_d(\mathbb{K})=\mathsf{K}_d\exp(\mathfrak{a}^{+})\mathsf{K}_d,$$ where $\mathfrak{a}^{+}=\big\{\textup{diag}(a_1,\ldots,a_d):a_1\geq a_2\geq \cdots \geq a_d\big\}$ and every $g\in \mathsf{GL}_d(\mathbb{K})$ decomposes as \begin{align}\label{decomp} g=k_g\textup{diag}\big(\sigma_1(g),\ldots,\sigma_d(g)\big)k_g', \ \ k_g,k_g'\in \mathsf{K}_d.\end{align} Given $1\leq i\leq d-1$, $g\in \mathsf{GL}_d(\mathbb{K})$ is called {\em $i$-proximal} if $\lambda_i(g)>\lambda_{i+1}(g)$. A subgroup $\mathsf{G}<\mathsf{GL}_d(\mathbb{K})$ (resp. a linear representation) is called $i$-proximal if it (resp. its image) contains an $i$-proximal matrix. We say that $g$ has a {\em gap of index $i$} if $\sigma_i(g)>\sigma_{i+1}(g)$. In this case, writing $g$ as in (\ref{decomp}), the $i$-plane $$\Xi_{i}(g):=k_{g}\langle e_1,\ldots, e_i\rangle$$ is well-defined and does not depend on the choice of $k_g,k_g'\in \mathsf{K}_d$.

We equip the projective space $\mathbb{P}(\mathbb{K}^d)$ with the angle metric $d_{\mathbb{P}}$ given by the formula $$d_{\mathbb{P}}([v_1],[v_2])=\sqrt{1-|\langle v_1,v_2\rangle|^2}, \ \ ||v_1||=||v_2||=1.$$ For a $(d-1)$-hyperplane $V=w\langle e_1,\ldots, e_{d-1}\rangle$, $w\in \mathsf{K}_d$, and a line $[v]\in \mathbb{P}(\mathbb{K}^d)$, their distance is $$\textup{dist}_{\mathbb{P}(\mathbb{K}^d)}\left([v],\mathbb{P}(V)\right)=\inf_{x\in \mathbb{P}(V)}d_{\mathbb{P}}([v],x)=\frac{1}{||v||}|\langle v,w e_d\rangle |.$$

We will need the following standard estimate.

\begin{lemma}\label{bound-hyperplane} Let $g \in \mathsf{GL}_d(\mathbb{K})$ be a matrix with a gap of index $1$. For $x\in \mathbb{P}(\mathbb{K}^d)\smallsetminus \mathbb{P}(\Xi_{d-1}(g^{-1}))$ the following estimate holds: $$d_{\mathbb{P}}\big(gx, \Xi_1(g)\big)\leq \frac{\sigma_2(g)}{\sigma_1(g)} \textup{dist}_{\mathbb{P}(\mathbb{K}^d)}\big(x,\Xi_{d-1}(g^{-1})\big)^{-1}.$$\end{lemma}

\begin{proof} Let us write $g=k_g\textup{diag}(\sigma_1(g),\ldots,\sigma_d(g))k_g'$, $k_g,k_g'\in \mathsf{K}_d$, where $\Xi_1(g)=[k_ge_1]$ and $\Xi_{d-1}(g^{-1})=(k_{g}')^{-1}\langle e_2,\ldots,e_d\rangle$. If we write $x=[v]$ for some unit vector $v\in \mathbb{K}^d$, we have $\textup{dist}_{\mathbb{P}(\mathbb{K}^d)}([v],\Xi_{d-1}(g^{-1}))=|\langle k_g'v,e_1\rangle|$ and by the definition of the metric $d_{\mathbb{P}}$, \begin{align*}d_{\mathbb{P}}\big([gv],\Xi_1(g) \big)^2= \frac{\sum_{i=2}^{d}\sigma_i(g)^2 |\langle k_g'v,e_i\rangle|^2}{\sum_{i=1}^{d}\sigma_i(g)^2|\langle k_g'v,e_i\rangle|^2}&\leq \frac{\sigma_2(g)^2}{\sigma_1(g)^2}\frac{1}{|\langle k_g'v,e_1\rangle|^{2}}.\end{align*} This finishes the proof of the lemma.\end{proof}

For $C\subset \mathbb{P}(\mathbb{K}^d)$ denote by $\mathcal{N}_{\theta}(C):=\bigcup_{x\in C}\{y:d_{\mathbb{P}}(x,y)<\theta\}$ the $\theta$-neighbourhood of $C$ with respect to $d_{\mathbb{P}}$. We will also use the following observation.

\begin{lemma}\label{close-I} Fix $0< \theta, \theta'<1$. There is $\kappa>0$ such that for every $g\in \mathsf{GL}_d(\mathbb{K})$ with $||g-I_d||<\kappa$ and any compact subset $C\subset \mathbb{P}(\mathbb{K}^d)$, then $g^{\pm 1}\mathcal{N}_{\theta}(C)\subset \mathcal{N}_{\theta+\theta'}(C).$   \end{lemma} 

\begin{proof} We argue by contradiction. If such $\kappa>0$ does not exist, there are sequences $(g_n)_{n\in \mathbb{N}}\subset \mathsf{GL}_d(\mathbb{K})$ with $\lim_ng_n=I_d$, $C_n\subset \mathbb{P}(\mathbb{K}^d)$ and $(y_n)_{n\in \mathbb{N}}\subset \mathcal{N}_{\theta}(C_n)$ such that $g_ny_n\notin \mathcal{N}_{\theta+\theta'}(C_n)$. In particular, we have that $d_{\mathbb{P}}(g_ny_n,y_n)\geq \textup{dist}_{\mathbb{P}(\mathbb{K}^d)}(g_ny_n, C_n)-\textup{dist}_{\mathbb{P}(\mathbb{K}^d)}(y_n,C_n)\geq \theta'$. This clearly contradicts the fact that $(d_{\mathbb{P}}(g_ny_n,y_n))_{n\in \mathbb{N}}$ has a subsequence converging to $0$. The lemma follows. \end{proof}

\subsection{Quasi-isometric embeddings and Anosov representations.} \label{definition} Let $\Gamma$ be a finitely generated group. Fix a left invariant metric $d_{\Gamma}:\Gamma \times \Gamma \rightarrow \mathbb{R}_{+}$, induced by a finite generating subset of $\Gamma$, and denote by $|\cdot|_{\Gamma}:\Gamma \rightarrow \mathbb{R}_{+}$, $|\gamma|_{\Gamma}=d_{\Gamma}(\gamma,1)$, $\gamma \in \Gamma$, the word length function of $d_{\Gamma}$. The group $\Gamma$ is called {\em hyperbolic} if its Cayley graph, equipped with $d_\Gamma$, is a Gromov hyperbolic space. We refer to \cite{Gromov} for more background on hyperbolic groups.

A linear representation $\rho:\Gamma \rightarrow \mathsf{GL}_d(\mathbb{K})$ is called a {\em quasi-isometric embedding} if the the orbit map of $\rho$, $\tau_{\rho}:\Gamma \rightarrow \mathsf{GL}_d(\mathbb{K})/\mathsf{K}_d$, $\tau_{\rho}(\gamma)=\rho(\gamma)\mathsf{K}_d$, is a quasi-isometric embedding, where the target symmetric space $\mathsf{GL}_d(\mathbb{K})/\mathsf{K}_d$ is equipped with the Killing metric. Equivalently, $\rho$ is a quasi-isometric embedding if and only if there exist $R, \kappa>0$ such that $$\frac{\sigma_1(\rho(\gamma))}{\sigma_d(\rho(\gamma))}\geq  e^{\kappa|\gamma|_{\Gamma}-R} \ \  \forall \gamma \in \Gamma.$$

 Anosov representations were introduced in \cite{Labourie}, and have been extensively studied since then, see for example the papers \cite{GW, KLP2, GGKW, BPS, DGK, Zimmer}. More recently, it was established that cubulated hyperbolic groups admit Anosov representations in higher rank \cite{DFWZ}.

We shall provide the following definition of Anosov representations, equivalent to Labourie's original dynamical definition, established by Kapovich--Leeb--Porti \cite{KLP2} and independently by Bochi--Potrie--Sambarino \cite{BPS}.

\begin{definition}\label{Anosov-def}\textup{(\cite{BPS, KLP2})} Let $d,k\in \mathbb{N}_{\geq 2}$, $1\leq k \leq d-1$ and $\Gamma$ a finitely generated group. A representation $\rho:\Gamma \rightarrow \mathsf{GL}_d(\mathbb{K})$  is {\em $k$-Anosov} and $\Gamma$ is hyperbolic if and only if there exist $C,\alpha>0$ such that $$\frac{\sigma_{k}(\rho(\gamma))}{\sigma_{k+1}(\rho(\gamma))} \geqslant e^{\alpha |\gamma|_{\Gamma}-C} \ \ \forall \gamma \in \Gamma.$$ \end{definition} 

In contrast to rank one Lie groups, for a representation in $\mathsf{GL}_d(\mathbb{K})$, $d\geq 3$, the property of being Anosov (for some $k$) is much stronger than being a quasi-isometric embedding (e.g. see the example in \cite[Prop. A1]{GGKW}). There are also previously known examples of quasi-isometric embeddings of hyperbolic groups in $\mathsf{SL}_{d}(\mathbb{R})$, $d\geq 4$, which fail to be \hbox{limits of Anosov representations \cite{Tso-limit}.} Examples of Zariski dense representations of free groups in $\mathsf{SL}_3(\mathbb{R})$, which are inapproximable by Anosov representations and are quasi-isometric embeddings (but not robust quasi-isometric embeddings), were exhibited in \cite[Thm. 1.2]{CPL24}. 

 From now on, $\Gamma$ is assumed to be a hyperbolic group and $\partial_{\infty}\Gamma$ is its Gromov boundary. Every $k$-Anosov representation $\rho:\Gamma \rightarrow \mathsf{GL}_d(\mathbb{K})$ admits a unique pair of continuous $\rho$-equivariant maps $$(\xi_{\rho}^{k},\xi_{\rho}^{d-k}): \partial_{\infty}\Gamma \rightarrow \mathsf{Gr}_{k}(\mathbb{K}^d)\times \mathsf{Gr}_{d-k}(\mathbb{K}^d)$$ called the {\em Anosov limit maps of $\rho$}. The limit maps of a $k$-Anosov representation can be recovered from the planes $\Xi_k$ and $\Xi_{d-k}$, i.e they satisfy the so called {\em Cartan property}. More precisely, for every infinite sequence $(\gamma_n)_{n\in \mathbb{N}}\subset \Gamma$ converging in $\partial_{\infty}\Gamma$, $$\xi_{\rho}^{i}(\lim_{n\rightarrow \infty}\gamma_n)=\lim_{n \rightarrow \infty} \Xi_{i}(\rho(\gamma_n)) \ \ i\in \{k,d-k\},$$ see \cite[Thm. 5.3]{GGKW}. For more  background on Anosov representations we refer to \cite{Canary}.

The following lemma is an immediate consequence of the Cartan property of the Anosov limit maps.

\begin{lemma}\label{close} Let  $\rho:\Gamma \rightarrow \mathsf{SL}_m(\mathbb{K})$ be a faithful $1$-Anosov representation. For every $0<\theta<1$ there exists a finite subset $F\subset \Gamma$ such that for every $\gamma\in \Gamma\smallsetminus F$: $$\textup{dist}_{\mathbb{P}(\mathbb{K}^d)}\big(\Xi_1(\rho(\gamma)), \xi_{\rho}^1(\partial_{\infty}\Gamma)\big)\leq \theta,\  \mathbb{P}(\Xi_{d-1}(\rho(\gamma)))\subset \mathcal{N}_{\theta}\Big(\bigcup_{\eta\in \partial_{\infty}\Gamma}\mathbb{P}(\xi_{\rho}^{d-1}(\eta))\Big).$$\end{lemma}

We close this section with the following ping-pong lemma for quasi-isometrically embedded linear groups acting on the projective space.

\begin{lemma}\label{pingpong} Let $\Delta_1,\Delta_2,\ldots, \Delta_{\ell}<\mathsf{SL}_m(\mathbb{K})$, $\ell\geq 3$, be finitely generated quasi-isometrically embedded subgroups. Fix $|\cdot|_{\Delta_i}:\Delta_i\rightarrow \mathbb{R}_{+}$ a word length function on $\Delta_i$. Suppose that $\mathcal{C}_1, \mathcal{C}_2, \mathcal{M}_2,\ldots, \mathcal{M}_{\ell}$ are non-empty subsets of $\mathbb{P}(\mathbb{K}^m)$, where $\mathcal{M}_j\subset \mathcal{C}_1$ for every $j=2,\ldots, \ell$, $\mathcal{M}_i\cap \mathcal{M}_j=\emptyset$ for $i\neq j$ and there is $\alpha>0$ with the property:\\
\noindent \textup{(i)} $\gamma \mathcal{C}_1\subset \mathcal{C}_2$ and $||\gamma v||\geq e^{\alpha|\gamma|_{\Delta_1}}||v||$ for every $\gamma\in \Delta_1\smallsetminus \{1\}$, $[v]\in \mathcal{C}_1$,\\
\noindent \textup{(ii)} $\gamma (\mathcal{C}_2\cup \mathcal{M}_j)\subset \mathcal{M}_i$ and $||\gamma \omega||\geq e^{\alpha|\gamma|_{\Delta_i}}||\omega||$ for every $\gamma\in \Delta_i\smallsetminus \{1\}$, $[\omega]\in \mathcal{C}_2\cup \mathcal{M}_j$, $2\leq i\neq j\leq \ell$.\\
\noindent Then $\langle \Delta_1,\Delta_2,\ldots, \Delta_{\ell}\rangle$ is isomorphic to the free product $\Delta_1\ast \Delta_2\ast \cdots \ast \Delta_{\ell}$ and its inclusion in $\mathsf{SL}_m(\mathbb{K})$ is a quasi-isometric embedding.\end{lemma}

\begin{proof} Let $\Delta:=\Delta_1\ast \cdots \ast \Delta_{\ell}$ and $|\cdot|_{\Delta}:\Delta \rightarrow \mathbb{R}_{+}$ be the word length function restricting to $|\cdot|_{\Delta_i}$ on $\Delta_i$. We set $\Delta_2':=\langle \Delta_2,\ldots ,\Delta_{\ell}\rangle$. By using condition (ii) on $\mathcal{C}_2,\mathcal{M}_i$ and induction on $n\in \mathbb{N}$, we check that for any product of the form $\gamma_1\cdots \gamma_n\in \Delta_2'$, where $\gamma_i\neq I_m$ and no consecutive $\gamma_i$ lie in the same $\Delta_i$, every $[\omega]\in \mathcal{C}_2$ we have that \begin{align}\label{ping-pong1}||\gamma_1\cdots \gamma_n \omega|| \geq e^{\alpha\sum_{i=1}^{n}|\gamma_i|_{\Delta}}||\omega||\end{align} and $\gamma_1\cdots \gamma_n \omega\in \mathcal{M}_j$, where $j\in \{2,\ldots,\ell\}$ is the unique index with $\gamma_1\in \Delta_j$. This shows $\gamma_1\cdots \gamma_n \neq I_{m}$, hence, by (\ref{ping-pong1}), $\Delta_2'$ is quasi-isometrically embedded in $\mathsf{SL}_m(\mathbb{K})$ and isomorphic to the free product $\Delta_2 \ast \cdots \ast \Delta_{\ell}$. Now since $(\Delta_2'\smallsetminus \{I_m\})\mathcal{C}_2\subset \mathcal{C}_1$, by using (i) and (\ref{ping-pong1}) and induction on $r\in \mathbb{N}$, we deduce that for every product $\delta_1\cdots \delta_r\in \langle \Delta_1,\Delta_2'\rangle$, where $\delta_i\neq I_m$ and no consecutive $\delta_i$ lie in $\Delta_1$ or $\Delta_2'$, we have \begin{align}\label{ping-pong2} \big|\big|\delta_1\cdots \delta_rv\big|\big| \geq e^{\alpha\sum_{i=1}^{r}|\delta_i|_{\Delta}}||v||,\end{align} where $[v]\in \mathcal{C}_1$ if $\delta_r\in \Delta_1$ and $[v]\in \mathcal{C}_2$ if $\delta_r \in \Delta_2'$. In particular, $\sigma_1(\delta_1\cdots \delta_r)\geq e^{\alpha \sum_{i=1}^r|\delta|_{\Delta_i}}$, $\delta_1\cdots \delta_r\neq I_m$ and hence $\langle \Delta_1,\Delta_2\rangle$ is isomorphic to $\Delta_1\ast \Delta_2'$. Since $\Delta_1,\Delta_2'< \mathsf{SL}_d(\mathbb{K})$ are quasi-isometric embedded, (\ref{ping-pong2}) shows the natural inclusion $\Delta_1\ast \Delta_2'\rightarrow \mathsf{SL}_m(\mathbb{K})$ is a quasi-isometric embedding. This concludes the proof of the lemma.\end{proof}

\section{Quasi-isometrically embedded free products of Anosov groups}
Denote by $\mathcal{F}_{1,m-1}(\mathbb{K})\subset \mathbb{P}(\mathbb{K}^m)\times \mathsf{Gr}_{m-1}(\mathbb{K}^d)$ the space of $(1,m-1)$ partial flags, equipped with the induced metric $d_{\mathbb{P}}\times d_{\mathbb{P}}$, where $d_{\mathbb{P}}$ is the metric on $ \mathsf{Gr}_{m-1}(\mathbb{K}^m)\simeq  \mathbb{P}((\mathbb{K}^d)^{\ast})$. Recall that $(z_1,W_1),(z_2,W_2)\in \mathcal{F}_{1,m-1}(\mathbb{K})$ are transverse if $z_1\notin W_2, z_2\notin W_1$. In this section, we prove the following theorem which is crucial for the construction of the examples in Theorem \ref{mainthm}.

\begin{theorem} \label{stable1} Let $\Gamma<\mathsf{SL}_d(\mathbb{K})$ be a $1$-Anosov subgroup such that there is a flag $y\in \mathcal{F}_{1,d-1}(\mathbb{K})$ which is transverse to the limit set of $\Gamma$ in $\mathcal{F}_{1,d-1}(\mathbb{K})$. Let $p,\ell \in \mathbb{N}_{\geq 2}$, $r\in \mathbb{Z}_{\geq 0}$ and consider the representations $\rho_{p,r}:\Gamma\rightarrow \mathsf{SL}_{pd+r}(\mathbb{K})$ and $\psi_{p,r}:\Gamma \rightarrow \mathsf{SL}_{pd+r}(\mathbb{K})$ defined as follows: \begin{align}\label{rep-freeproduct} \rho_{p,r}(\gamma):=\begin{pmatrix}
I_r &  &  & \\ 
 & \gamma &  & \\ 
 &  & \ddots  & \\ 
 &  &  & \gamma
\end{pmatrix}, \ \psi_{p,r}(\gamma):=\begin{pmatrix}
I_r &  &  & \\ 
 & I_{d} &  & \\ 
 &  & I_{(p-2)d}  & \\ 
 &  &  & \gamma\end{pmatrix} \ \ \gamma \in \Gamma.\end{align} There exist finite-index subgroups $\Gamma_1,\Gamma_2,\ldots, \Gamma_{\ell}<\Gamma$, $w_2,\ldots, w_{\ell}\in \mathsf{SL}_{pd+r}(\mathbb{K})$ and $\zeta>0$ with the property: for every $g_1,g_2,\ldots, g_{\ell}\in \mathsf{GL}_{pd+r}(\mathbb{K})$ with $||g_i-I_{pd+r}||<\zeta$, the unique representation $\phi_{p,r}(\cdot, g_1,\ldots,g_{\ell}):\Gamma_1 \ast \Gamma_2 \ast \cdots\ast \Gamma_{\ell}\rightarrow \mathsf{SL}_{pd+r}(\mathbb{K})$ satisfying \begin{align}\label{def-phi(p,r)}\phi_{p,r}(\gamma,g_1,\ldots,g_{\ell}):=\left\{\begin{matrix}
 g_1\rho_{p,r}(\gamma)g_1^{-1}, & \gamma \in \Gamma_1  \\ 
(g_iw_i)\psi_{p,r}(\gamma)(g_iw_i)^{-1}, &  \gamma \in \Gamma_i,\ i=2,\ldots, \ell 
\end{matrix}\right.\end{align} is a quasi-isometric embedding.\end{theorem}

\begin{proof} In order to prove the theorem, we shall use a ping-pong combination of the representation $\rho_{p,r}$ with certain conjugates of $\psi_{p,r}$, restricted to appropriate finite-index subgroups of $\Gamma$ respectively.

Let $q:=pd+r$ and $(\xi^{1},\xi^{d-1}):\partial_{\infty}\Gamma \rightarrow \mathbb{P}(\mathbb{K}^d)\times \mathsf{Gr}_{d-1}(\mathbb{K}^d)$ be the Anosov limit maps of $\Gamma$. By assumption, there is a $(1,d-1)$ flag $y=([v_0],W)$, $||v_0||=1$, transverse to the limit set $\big\{(\xi^1(\eta), \xi^{d-1}(\eta)):\eta \in \partial_{\infty}\Gamma\big\}$ of $\Gamma$ in $\mathcal{F}_{1,d-1}(\mathbb{K})$ and define $0<\varepsilon<\frac{1}{11}$, $$\varepsilon:= \frac{1}{11} \inf_{\eta \in \partial_{\infty}\Gamma}\min \big\{\textup{dist}_{\mathbb{P}(\mathbb{K}^d)}(\xi^1(\eta), \mathbb{P}(W)),\  \textup{dist}_{\mathbb{P}(\mathbb{K}^d)}([v_0], \mathbb{P}(\xi^{d-1}(\eta)) )\big \}>0.$$

Since $\Gamma<\mathsf{SL}_d(\mathbb{K})$ is $1$-Anosov and quasi-isometrically embedded, by Lemma \ref{close}, we may pass to a finite-index subgroup $\Gamma_1<\Gamma$ (depending only on $\varepsilon>0$), fix a word metric $|\cdot|_{\Gamma_1}:\Gamma_1\rightarrow \mathbb{R}_{+}$ and $\alpha_1>0$ such that for every $\gamma \in \Gamma_1\smallsetminus \{1\}$ the following hold: \begin{align}\nonumber \Xi_1(\gamma)\in \mathcal{N}_{\varepsilon}(\xi^1(\partial_{\infty}\Gamma)),& \  \mathbb{P}(\Xi_{d-1}(\gamma))\subset \mathcal{N}_{\varepsilon}\Big(\bigcup_{\eta\in \partial_{\infty}\Gamma}\mathbb{P}(\xi^{d-1}(\eta))\Big)\\
\label{stable1-eq1} \textup{dist}_{\mathbb{P}(\mathbb{K}^d)}\big(\Xi_1(\gamma), \mathbb{P}(W)\big)\geq 10\varepsilon,\  & \textup{dist}_{\mathbb{P}(\mathbb{K}^d)}\big([v_0], \mathbb{P}(\Xi_{d-1}(\gamma))\big)\geq 10\varepsilon\\ \label{stable1-eq1'}\sigma_1(\gamma)\geq e^{\alpha_1|\gamma|_{\Gamma_1}}\geq \varepsilon^{-5}, \ &\ \frac{\sigma_1(\gamma)}{\sigma_2(\gamma)}\geq \varepsilon^{-5} \ \ \forall \gamma \in \Gamma_1\smallsetminus \{1\}.\end{align}

Having the previous notation and the finite-index subgroup $\Gamma_1<\Gamma$, we proceed with three steps.
\medskip

\noindent {\em Step 1: Defining ping-pong sets for the restriction of $\rho_{p,r}$ on $\Gamma_1$.}
\medskip

Let us set $v_0':=\big(0,\ldots,0, v_0)^{t}\in \mathbb{K}^{q}, W':= \mathbb{K}^{(p-1)d+r}\oplus W$ and consider the sets $$\mathcal{C}_{1}:=B_{2\varepsilon}([v_0']), \ \mathcal{C}_2:=\mathbb{P}(\mathbb{K}^{q})\smallsetminus \mathcal{N}_{4\varepsilon^2}(\mathbb{P}(W')).$$ 

By Lemma \ref{close-I} we choose $0<\varepsilon_1<\varepsilon^5$ such that for every $g\in \mathsf{GL}_q(\mathbb{K})$ with $||g-I_d||<\varepsilon_1$, $\sigma_1(g^{\pm 1})\leq 2$ and for every compact subset $C\subset \mathbb{P}(\mathbb{K}^q)$, $$g^{\pm 1}C\subset \mathcal{N}_{\varepsilon^5}(C), \ g^{\pm 1}\mathcal{N}_{2\varepsilon}(C)\subset \mathcal{N}_{3\varepsilon}(C).$$

\begin{claim}\label{Claim-1} For every $\gamma \in \Gamma_1\smallsetminus \{1\}$ and $g\in \mathsf{GL}_{q}(\mathbb{K})$ with $||g-I_{q}||<\varepsilon_1$, $g\rho_{p,r}(\gamma)g^{-1}\mathcal{C}_1\subset \mathcal{C}_2$. Moreover for every $[\omega]\in \mathcal{C}_1$, $$\big|\big|g\rho_{p,r}(\gamma)g^{-1}\omega\big|\big| \geq e^{\frac{\alpha_1}{2}|\gamma|_{\Gamma_1}} ||\omega||.$$\end{claim}

\begin{proof}[Proof of Claim \ref{Claim-1}] Let $g\in \mathsf{GL}_{q}(\mathbb{K})$ with $||g-I_{q}||<\varepsilon_1$ so that $g^{-1}B_{2\varepsilon}([v_0'])\subset B_{3\varepsilon}([v_0'])$. Fix $[\omega] \in B_{2\varepsilon}([v_0'])$, where $||\omega||=1$, $[g^{-1}\omega]\in B_{3\varepsilon}([v_0'])$. By the definition of $d_{\mathbb{P}}$ we may write \begin{align*} \frac{g^{-1}\omega}{||g^{-1}\omega||}&=\big(\omega_{1},\omega_{2})^t \in \mathbb{K}^{(p-1)d+r}\oplus \mathbb{K}^d,\ ||\omega_{1}||\leq 3\varepsilon,\ ||\omega_2||\geq 1-3\varepsilon,\\ &d_{\mathbb{P}}([\omega_{2}], [v_0])\leq \frac{d_{\mathbb{P}}([g^{-1}\omega],[v_0'])}{||\omega_{2}||}\leq \frac{3\varepsilon}{||\omega_{2}||}\leq 6\varepsilon.\end{align*} Fix $\gamma \in \Gamma_1\smallsetminus\{1\}$. By (\ref{stable1-eq1}) we have, \begin{align}\label{stable1-Xi-gamma} \textup{dist}_{\mathbb{P}(\mathbb{K}^d)}\big([\omega_2], \Xi_{d-1}(\gamma^{-1})\big)&\geq \textup{dist}_{\mathbb{P}(\mathbb{K}^d)}([v_0], \Xi_{d-1}(\gamma^{-1})\big)-d_{\mathbb{P}}\big([v_0],[\omega_2])\geq 4\varepsilon,\ \textup{hence}\\ \nonumber ||\gamma \omega_2||&\geq \sigma_1(\gamma)\textup{dist}_{\mathbb{P}(\mathbb{K}^q)}([\omega_2],\mathbb{P}(\Xi_{d-1}(\gamma^{-1})))||\omega_2||\\ \label{stable1-gamma-omega2}&\geq 4\varepsilon(1-3\varepsilon)\sigma_1(\gamma). \end{align} Therefore, we obtain the bounds: \begin{align*}\big| \big| g\rho_{p,r}(\gamma)g^{-1}\omega \big|\big|&\geq \frac{||g^{-1}\omega||}{\sigma_1(g^{-1})}\Big| \Big| \rho_{p,r}(\gamma)\frac{g^{-1}\omega}{||g^{-1}\omega \big|\big|}\Big|\Big|=\frac{||g^{-1}\omega||}{\sigma_1(g^{-1})}\big| \big| \rho_{p,r}(\gamma)(\omega_1,\omega_2)^t\big|\big| \\ &\geq \frac{1}{4} \big|\big| \big(\textup{diag}(I_r,\gamma,\ldots, \gamma)\omega_1,\gamma \omega_2\big)^t\big|\big|\\ &\geq \frac{1}{4} ||\gamma \omega_2|| \geq \frac{||\omega_2||}{4}\sigma_1(\gamma) \textup{dist}_{\mathbb{P}(\mathbb{K}^d)}\big([\omega_2], \mathbb{P}(\Xi_{d-1}(\gamma^{-1}))\big)\\  &\geq \varepsilon(1-3\varepsilon)\sigma_1(\gamma) \geq \sqrt{\sigma_1(\gamma)}\\ &\geq e^{\frac{\alpha_1}{2}|\gamma|_{\Gamma_1}}. \end{align*} Moreover, with respect to the decomposition $\mathbb{K}^{q}=\mathbb{K}^{(p-1)d+r}\oplus \mathbb{K}^d$ we can write \begin{align}\label{stable1-eq2} \frac{\rho_{p,r}(\gamma)g^{-1}\omega}{||\rho_{p,r}(\gamma)g^{-1}\omega||}=\frac{||g^{-1}\omega||\textup{diag}(I_r, \gamma,\ldots,\gamma)\omega_1}{||\rho_{p,r}(\gamma)g^{-1}\omega||}+\frac{||g^{-1}\omega||\gamma \omega_2}{||\rho_{p,r}(\gamma)g^{-1}\omega||}.\end{align}  Thus, since $\frac{\sigma_1(\gamma)}{\sigma_2(\gamma)}\geq \varepsilon^{-5}$ by (\ref{stable1-eq1'}), Lemma \ref{bound-hyperplane}, combined with (\ref{stable1-Xi-gamma}), implies the bound \begin{align}\label{stable1-eq3} d_{\mathbb{P}}\big([\gamma\omega_2],\Xi_1(\gamma)\big)\leq \frac{1}{4\varepsilon}\frac{\sigma_2(\gamma)}{\sigma_1(\gamma)}\leq \varepsilon. \end{align} Using (\ref{stable1-gamma-omega2}), (\ref{stable1-eq2}), (\ref{stable1-eq3}) and Lemma \ref{close-I}, we successively have the bounds: \begin{align*}\textup{dist}_{\mathbb{P}(\mathbb{K}^{q})}\big([g\rho_{p,r}(\gamma)g^{-1}\omega], \mathbb{P}(W')\big) &\geq \textup{dist}_{\mathbb{P}(\mathbb{K}^{q})}\big(\rho_{p,r}(\gamma)[g^{-1}\omega], \mathbb{P}(W')\big)-d_{\mathbb{P}}\big(g[\rho_{p,r}(\gamma)g^{-1}\omega],[\rho_{p,r}(\gamma)g^{-1}\omega]\big)\\ &\geq \textup{dist}_{\mathbb{P}(\mathbb{K}^{q})}\big(\rho_{p,r}(\gamma)[g^{-1}\omega], \mathbb{P}(W')\big)-\varepsilon^{5}\\ &= \textup{dist}_{\mathbb{P}(\mathbb{K}^{q})}\Bigg(\Bigg[\frac{||g^{-1}\omega||\textup{diag}(I_r,\gamma,\ldots,\gamma)\omega_1+||g^{-1}\omega||\gamma \omega_2}{||\rho_{p,r}(\gamma)g^{-1}\omega||}\Bigg], \mathbb{P}(\mathbb{K}^{(p-1)d}\oplus W)\Bigg)-\varepsilon^{5}\\  &=\frac{||g^{-1}\omega||\cdot ||\gamma \omega_2||}{||\rho_{p,r}(\gamma)g^{-1}\omega||} \textup{dist}_{\mathbb{P}(\mathbb{K}^d)}\big( [\gamma \omega_2], \mathbb{P}(W)\big)-\varepsilon^{5}\\ &\geq \frac{4\varepsilon(1-3\varepsilon)\sigma_1(\gamma)||g^{-1}\omega||}{||\rho_{p,r}(\gamma)||\cdot ||g^{-1}\omega||}\Big(\textup{dist}_{\mathbb{P}(\mathbb{K}^d)}\big(\Xi_{1}(\gamma), \mathbb{P}(W)\big)-d_{\mathbb{P}}\big([\gamma \omega_2], \Xi_{1}(\gamma)\big)\Big)-\varepsilon^{5}\\ &\geq \frac{9\varepsilon^2}{2}-\varepsilon^{5}> 4\varepsilon^2, \end{align*} thus $[g\rho_{p,r}(\gamma)g^{-1}\omega]\in \mathcal{C}_2$. This completes the proof of Claim \ref{Claim-1}. \end{proof}

Now we proceed with the second step by defining ping-pong sets associated to $\psi_{p,r}$.
\medskip

\noindent {\em Step 2. Defining ping-pong sets for certain conjugates of $\psi_{p,r}$.}
\medskip

Let us recall the representation $\psi_{p,r}:\Gamma_1\rightarrow \mathsf{SL}_{q}(\mathbb{K})$, $q=pd+r$, from (\ref{rep-freeproduct}). By Definition \ref{Anosov-def}, $\psi_{p,r}$ is $1$-Anosov with limit maps $(\xi_{\psi_{p,r}}^1,\xi_{\psi_{p,r}}^{q-1}):\partial_{\infty}\Gamma_1\rightarrow \mathbb{P}(\mathbb{K}^{q})\times \mathsf{Gr}_{q-1}(\mathbb{K}^{q})$, $$\xi_{\psi_{p,r}}^1(\eta)=\xi^1(\eta), \ \xi_{\psi_{p,r}}^{q-1}(\eta)=\mathbb{K}^{(p-1)d+r}\oplus \xi^{d-1}(\eta),\ \eta\in \partial_{\infty}\Gamma.$$ 

The flag $y':=([v_0'], W')$, $v_0'=(0,\ldots,0,v_0)^t$, $W'=\mathbb{K}^{(p-1)d+r}\oplus W$, is transverse to the limit set of $\psi_{p,r}(\Gamma_1)$ in $\mathcal{F}_{1,q-1}(\mathbb{K})$. Now choose $0<40\theta<\varepsilon^5$, where $\theta>0$ depends only on $y'$ and $\varepsilon>0$, and pairwise transverse flags $(x_2,V_2),\ldots, (x_{\ell},V_{\ell})\in \mathcal{F}_{1,q-1}(\mathbb{K})$ in a $10\theta$-neighbourhood (with respect to the fixed metric on $\mathcal{F}_{1,q-1}(\mathbb{K})$) of $y'$, each of them transverse to the limit set of $\psi_{p,r}(\Gamma_1)$ in $\mathcal{F}_{1,d-1}(\mathbb{K})$, and $\textup{dist}_{\mathbb{P}(\mathbb{K}^q)}(x_i,\mathbb{P}(W'))\geq 2\theta$ for every $i$.  In particular, we may further choose $0<\epsilon< 10^{-2}\theta$ small enough, depending only on $\{(x_i,V_i)\}_{i=2}^{\ell}$ and $\theta>0$, such that the following hold true:
\medskip

\noindent \textup{(i)} $\mathcal{M}_i:=B_{2\epsilon}(x_i)\subset \mathcal{C}_1$ and $\mathcal{M}_i\cap \mathcal{M}_j=\emptyset$ for $2\leq i\neq j\leq \ell$,\\
\noindent \textup{(ii)} $\textup{dist}_{\mathbb{P}(\mathbb{K}^q)}(\mathcal{M}_j, \mathbb{P}(V_i))\geq 12\epsilon$, for $2\leq i\neq j\leq \ell$, \\
\noindent \textup{(iii)} $\textup{dist}_{\mathbb{P}(\mathbb{K}^q)}(\mathcal{M}_i, \mathbb{P}(W'))\geq \theta$ and $\mathbb{P}(V_{i})\subset \mathcal{N}_{11\theta}(\mathbb{P}(W'))$ for $i=2,\ldots, \ell$.
\medskip

We also need the following claim.

\begin{claim} \label{Claim-2} For every $i=2,\ldots, \ell$ there exists a finite-index subgroup $\Gamma_i<\Gamma_1$, a word metric $|\cdot|_{\Gamma_i}:\Gamma_i\rightarrow \mathbb{R}_{+}$, $\alpha_i>0$, $w_i\in \mathsf{SL}_q(\mathbb{K})$ and $\epsilon_1>0$ with the property: for every $j\neq i$, $\delta\in \Gamma_i\smallsetminus \{1\}$ and $h\in \mathsf{GL}_{q}(\mathbb{K})$ with $||h-I_{q}||<\epsilon_1$, $$h(w_i\psi_{p,r}(\delta)w_i^{-1})h^{-1} \big(\mathcal{C}_2\cup \mathcal{M}_j\big)\subset \mathcal{M}_i.$$  In addition, for every $[v]\in \mathcal{C}_2\cup \mathcal{M}_j$, \begin{align}\label{Claim-2-pingpong2}\big|\big| h(w_i\psi_{p,r}(\delta)w_i^{-1})h^{-1}v\big|\big|\geq e^{\frac{\alpha_i}{2}|\delta|_{\Gamma_i}} ||v||.\end{align}\end{claim}

\begin{proof}[Proof of Claim \ref{Claim-2}] Recall that for $i=2,\ldots, \ell$, $(x_i,V_i)$ is transverse to the limit set of $\psi_{p,r}(\Gamma_1)$. Hence, we may choose $w_i\in \mathsf{SL}_{q}(\mathbb{K})$, $\{1,q-1\}$-proximal  depending only on $\theta,\epsilon>0$ and with attarcting and repelling fixed flags sufficiently close to $(x_i,V_i)$ such that $$d_{\mathbb{P}}(w_i \xi_{\psi_{p,r}}^1(\eta), x_{i}) \leq \epsilon^3,\ 
 \mathbb{P}(w_i \xi_{\psi_{p,r}}^{q-1}(\eta))\subset \mathcal{N}_{\epsilon^3}(\mathbb{P}(V_{i})) \ \ \forall \eta\in \partial_{\infty}\Gamma_1.$$ In particular, by (iii), as $\mathbb{P}(V_i)\subset \mathcal{N}_{11\theta}(\mathbb{P}(W'))$ and $\epsilon<10^{-2}\theta$, we have the inclusions: \begin{align}\label{stable-inclusion1'}\mathcal{N}_{\frac{\epsilon}{2}}\big(w_i\xi_{\psi_{p,r}}^{1}(\partial_{\infty}\Gamma_1)\big)\subset B_{\epsilon}(x_{i}),\  \mathcal{N}_{\frac{\theta}{4}} \Big( \bigcup_{\eta\in \partial_{\infty}\Gamma_1} \mathbb{P}(w_i\xi_{\psi_{p,r}}^{q-1}(\eta)) \Big) \subset \mathcal{N}_{12\theta}(\mathbb{P}(W')).\end{align} Let $\psi_{p,r}^{(i)}:=w_i\psi_{p,r} w_i^{-1}$, $i=2,\ldots, \ell$. Since $\psi_{p,r}^{(i)}$ is $1$-Anosov, by Lemma \ref{close}, (\ref{stable-inclusion1'}) and the choices from (i), (ii), (iii), we may pass to a finite-index subgroup $\Gamma_i<\Gamma_1$, fix a word metric $|\cdot|_{\Gamma_i}:\Gamma_i\rightarrow \mathbb{R}_{+}$ (which is the restriction of $|\cdot|_{\Gamma_1}$) and choose $\alpha_i>0$ with the following properties: \begin{align}\label{stable1-inclusion1} \Xi_1(\psi^{(i)}_{p,r}(\delta))\in \mathcal{N}_{\frac{\epsilon}{4}}\big(w_i\xi_{\psi_{p,r}}^{1}(\partial_{\infty}\Gamma_1)\big)&\subset B_{\epsilon}(x_i) \ \ \forall  \delta\in \Gamma_i\smallsetminus \{1\},\\ 
\label{stable1-inclusion2} \mathcal{C}_{3i}:=\mathcal{N}_{\frac{\theta}{4}} \Big(\bigcup_{\delta \in \Gamma_i\smallsetminus \{1\}} \mathbb{P}\big(\Xi_{q-1}(\psi^{(i)}_{p,r}(\delta))\big)&\Big)\subset \mathcal{N}_{13\theta}(\mathbb{P}(W')), \\\ 
\label{stable1-bounds1''}   \mathbb{P}\big(\Xi_{q-1}(\psi_{p,r}^{(i)}(\delta)\big)\subset \mathcal{N}_{\epsilon}&(\mathbb{P}(V_i)), \  \forall \delta \in \Gamma_i\smallsetminus \{1\}\\
\label{stable1-bounds1'}   \textup{dist}_{\mathbb{P}(\mathbb{K}^q)}\big(B_{3\epsilon}(x_{j}), \mathbb{P}(\Xi_{q-1}(\psi_{p,r}^{(i)}&(\delta))\big)\geq 8\epsilon, \ i\neq j, \  \forall \delta \in \Gamma_i\smallsetminus \{1\} \\
\label{stable1-bounds1} \sigma_1\big(\psi^{(i)}_{p,r}(\delta)\big)\geq e^{\alpha_i |\delta|_{\Gamma_i}}\geq \epsilon^{-3},\  & \frac{\sigma_1(\psi^{(i)}_{p,r}(\delta))}{\sigma_2(\psi^{(i)}_{p,r}(\delta))}\geq \epsilon^{-3}  \ \ \forall \delta \in \Gamma_i\smallsetminus \{1\}.
\end{align} 

Fix $\delta \in \Gamma_i\smallsetminus \{1\}$ and $2\leq j\neq i\leq \ell $. If $x\in (\mathbb{P}(\mathbb{K}^{q})\smallsetminus \mathcal{C}_{3i})\cup B_{\epsilon}(x_{j})$, by the definition of the set $\mathcal{C}_{3i}$ above and (\ref{stable1-bounds1'}),  $\textup{dist}_{\mathbb{P}(\mathbb{K}^{q})}(x,\Xi_{q-1}(\psi^{(i)}_{p,r}(\delta^{-1})))\geq 8\epsilon$. Thus Lemma \ref{bound-hyperplane} and (\ref{stable1-inclusion1}) show that, \begin{align}\nonumber \textup{dist}_{\mathbb{P}(\mathbb{K}^{q})}\big(\psi^{(i)}_{p,r}(\delta)x, w_i\xi_{\psi_{p,r}}^{1}(\partial_{\infty}\Gamma_1) \big)&\leq d_{\mathbb{P}}\big(\psi^{(i)}_{p,r}(\delta)x, \Xi_1(\psi^{(i)}_{p,r}(\delta))\big)\\ &+\textup{dist}_{\mathbb{P}(\mathbb{K}^{q})}\big(\Xi_1(\psi^{(i)}_{p,r}(\delta)), w_i\xi_{\psi_{p,r}}^{1}(\partial_{\infty}\Gamma_1) \big)\\ 
\nonumber &\leq \frac{1}{8\epsilon}\frac{\sigma_2(\psi^{(i)}_{p,r}(\delta))}{\sigma_1(\psi^{(i)}_{p,r}(\delta))}+\frac{\epsilon}{4}<\frac{3\epsilon}{8}, \ \textup{hence by (\ref{stable-inclusion1'})}\\ \label{stable1-inclusion3} \psi^{(i)}_{p,r}(\delta) \big((\mathbb{P}(\mathbb{K}^{q})\smallsetminus \mathcal{C}_{3i})&\cup B_{\epsilon}(x_{j})\big) \subset \mathcal{N}_{\frac{\epsilon}{2}}\big(w_i\xi_{\psi_{p,r}}^1(\partial_{\infty}\Gamma_1)\big)\subset B_{\epsilon}(x_{i}).\end{align}

By Lemma \ref{close-I}, choose $0<\epsilon_1\ll 1$ such that for every $h\in \mathsf{GL}_{q}(\mathbb{K})$ with $||h-I_{q}||< \epsilon_1$ and $s,j=2,\ldots, \ell$, we have $$\mathcal{N}_{2\varepsilon^2}(\mathbb{P}(W'))\subset h^{-1}\mathcal{N}_{4\varepsilon^2}(\mathbb{P}(W')), \ h^{-1}\mathcal{M}_{j}\subset B_{3\epsilon}(x_{j}), \ h^{-1}\mathcal{M}_{s}\subset B_{3\epsilon}(x_{s})$$ and $d_{\mathbb{P}}(h^{\pm 1}z, z)\leq \epsilon$ for every $z\in \mathbb{P}(\mathbb{K}^q)$.

By using (\ref{stable1-inclusion1}), (\ref{stable1-inclusion2}), since $\mathcal{C}_{3i}\subset \mathcal{N}_{13\theta}(\mathbb{P}(W'))\subset  \mathcal{N}_{2\varepsilon^2}(\mathbb{P}(W'))$ and (\ref{stable1-inclusion3}), we conclude that \begin{align*}h\psi^{(i)}_{p,r}(\delta)h^{-1}\big( (\mathbb{P}(\mathbb{K}^{q})\smallsetminus \mathcal{N}_{4\varepsilon^2}(\mathbb{P}(W'))\cup \mathcal{M}_j\big)&\subset  h\psi^{(i)}_{p,r}(\delta)\big( (\mathbb{P}(\mathbb{K}^{q})\smallsetminus \mathcal{N}_{2\varepsilon^2}(\mathbb{P}(W')))\cup B_{3\epsilon}(x_{j})\big)\\ &\subset h \psi^{(i)}_{p,r}(\delta)\big((\mathbb{P}(\mathbb{K}^{q})\smallsetminus \mathcal{C}_{3i} )\cup B_{3\epsilon}(x_{j})\big)\\ &\subset h B_{\epsilon}(x_{i})\subset B_{2\epsilon}(x_{i})=\mathcal{M}_i. \end{align*}

This shows that for every $\delta \in \Gamma_i\smallsetminus \{1\}$ and $j\neq i$, $$h\psi^{(i)}_{p,r}(\delta)h^{-1}(\mathcal{C}_2 \cup \mathcal{M}_j)\subset \mathcal{M}_i.$$ This implies the first part of the claim.

 Now we check that (\ref{Claim-2-pingpong2}) holds. For this, let $[v]\in \mathcal{C}_2\cup \mathcal{M}_j$, $j\neq i$ and $\delta\in \Gamma_i\smallsetminus\{1\}$. By the choice of $\epsilon_1>0$, $d_{\mathbb{P}}([h^{-1}v],[v])\leq \epsilon$. If $[v]\in \mathcal{C}_2$, then $\textup{dist}_{\mathbb{P}(\mathbb{K}^{q})}([h^{-1}v],\mathbb{P}(W'))\geq 4\varepsilon^2-\epsilon\geq 3\varepsilon^2$. By (\ref{stable1-inclusion2}), $\mathbb{P}(\Xi_{q-1}(\psi_{p,r}^{(i)}(\delta^{-1})))\subset \mathcal{N}_{13\theta}(\mathbb{P}(W'))$, hence \hbox{$\textup{dist}_{\mathbb{P}(\mathbb{K}^{q})}([h^{-1}v],\mathbb{P}(\Xi_{q-1}(\psi_{p,r}^{(i)}(\delta^{-1})))\geq 3\varepsilon^2-13\theta\geq 2\varepsilon^2$.} If $[v]\in \mathcal{M}_j$, by (\ref{stable1-bounds1''}) and (ii) we have $$\textup{dist}_{\mathbb{P}(\mathbb{K}^{q})}\big([h^{-1}v],\mathbb{P}(\Xi_{q-1}(\psi^{(i)}_{p,r}(\delta^{-1}))\big)\geq \textup{dist}_{\mathbb{P}(\mathbb{K}^{q})}(\mathcal{M}_j,\mathbb{P}(V_i))-\epsilon-d_{\mathbb{P}}([h^{-1}v],[v])\geq 8\epsilon.$$ In each case, for every $[v]\in \mathcal{C}_2\cup \mathcal{M}_j$ and $\delta \in \Gamma_i\smallsetminus \{1\}$, we have the bound $$\textup{dist}_{\mathbb{P}(\mathbb{K}^{q})}\big([h^{-1}v],\mathbb{P}(\Xi_{q-1}(\psi_{p,r}^{(i)}(\delta^{-1}))\big)\geq 8\epsilon.$$  Hence, as $\sigma_1(\psi_{p,r}^{(i)}(\delta))\geq e^{\alpha_i |\delta|_{\Gamma_i}}\geq \epsilon^{-3}$ by (\ref{stable1-bounds1}), we obtain the lower bounds: \begin{align*} \big|\big| h\psi^{(i)}_{p,r}(\delta)h^{-1}v \big|\big|&\geq \frac{1}{\sigma_1(h^{-1})\sigma_1(h)} \Big|\Big| \psi^{(i)}_{p,r}(\delta)\frac{h^{-1}v}{||h^{-1}v||}\Big|\Big|\cdot ||v||\\ &\geq \frac{\sigma_1(\psi^{(i)}_{p,r}(\delta))}{4}\textup{dist}_{\mathbb{P}(\mathbb{K}^{q})}\big([h^{-1}v],\mathbb{P}(\Xi_{q-1}(\psi^{(i)}_{p,r}(\delta^{-1}))\big)||v||\\ & \geq  (2\epsilon)\sigma_1(\psi^{(i)}_{p,r}(\delta)) ||v||\geq \frac{2}{\sqrt{\epsilon}}\sqrt{\sigma_1(\psi_{p,r}^{(i)}(\delta))}||v|| \\ &\geq e^{\frac{\alpha_i}{2}|\delta|_{\Gamma_i}}||v|| . \end{align*} This finishes the proof of the claim.\end{proof} 

\noindent {\em Step 3. Completing the proof of Theorem \ref{stable1}.} Recall the constants $\varepsilon_1,\epsilon_1>0$ defined in Steps 1 and 2 respectively and set $\zeta:=\min\{\epsilon_1,\varepsilon_1\}>0$. For $g_1,\ldots, g_{\ell}\in \mathsf{GL}_{q}(\mathbb{K})$, with \hbox{$||g_i-I_{q}||<\zeta$}, $\phi_{p,r}(\cdot,g_1,\ldots,g_{\ell}):\Gamma_1\ast \cdots\ast \Gamma_{\ell} \rightarrow \mathsf{SL}_{q}(\mathbb{K})$ is the representation in (\ref{def-phi(p,r)}).  By Claim \ref{Claim-1} and Claim \ref{Claim-2}, for every $\gamma \in \Gamma_1\smallsetminus\{1\},  \delta\in \Gamma_i\smallsetminus \{1\}, [\omega]\in \mathcal{C}_1$ and $[v]\in \mathcal{C}_2\cup \mathcal{M}_j$, $2\leq i\neq j\leq \ell$, we have that:  \begin{align*}g_1\rho_{p,r}(\gamma)g_1^{-1}\mathcal{C}_1\subset \mathcal{C}_2, \ & \big|\big| g_1\rho_{p,r}(\gamma)g_1^{-1}\omega\big|\big|\geq e^{\frac{\alpha_1}{2}|\gamma|_{\Gamma_1}}||\omega|| ,\\
 g_iw_i\psi_{p,r}(\delta)(g_iw_i)^{-1}(\mathcal{C}_2\cup \mathcal{M}_j)\subset \mathcal{M}_i\subset \mathcal{C}_1, \ &
\big|\big| g_iw_i\psi_{p,r}(\delta)(g_iw_i)^{-1}v\big|\big|\geq e^{\frac{\alpha_i}{2}|\delta|_{\Gamma_2}}||v|| .
\end{align*} Thus by applying Lemma \ref{pingpong} for the groups $\phi_{p,r}(\Gamma_i, g_1,g_2,\ldots,g_{\ell})<\mathsf{SL}_{q}(\mathbb{K})$, $i=1,\ldots \ell$, and the ping-pong sets $\mathcal{C}_1,\mathcal{C}_2,\mathcal{M}_2,\ldots,\mathcal{M}_{\ell}\subset \mathbb{P}(\mathbb{K}^q)$ we conclude that $\phi_{p,r}(\cdot,g_1,\ldots,g_{\ell})$ is a quasi-isometric embedding. This concludes the proof of the theorem.\end{proof}

\begin{rmk}\normalfont{Note that Theorem \ref{stable1} does not follow from any known combination theorems for free products of $k$-Anosov groups (see \cite{DK, DKL}), since the representations $\rho_{p,r}:\Gamma \rightarrow \mathsf{SL}_{pd+r}(\mathbb{K})$ and $\psi_{p,r}:\Gamma \rightarrow \mathsf{SL}_{pd+r}(\mathbb{K})$, defined in (\ref{rep-freeproduct}), might fail to be Anosov for some common integer $1\leq k \leq \frac{pd+r}{2}$.}  \end{rmk}

For a hyperbolic group $\Gamma$, the space of representations $\textup{Hom}(\Gamma,\mathsf{SL}_{q}(\mathbb{K}))$ is equipped with the induced topology as a closed subset of the direct product of finitely many copies of $\mathsf{SL}_q(\mathbb{K})$, one for each generator of $\Gamma$. A linear representation $\phi:\Gamma \rightarrow \mathsf{SL}_q(\mathbb{K})$ is called {\em locally rigid} if there is a neighbourhood of $\phi$ in $\textup{Hom}(\Gamma,\mathsf{SL}_q(\mathbb{K}))$ \hbox{consisting only of conjugates of $\phi$.}

\begin{lemma}\label{non-rigid} Let $\Gamma<\mathsf{SL}_{d}(\mathbb{K})$ be an irreducible $1$-Anosov subgroup, $p,\ell\in \mathbb{N}_{\geq 2}$, $r\in \mathbb{Z}_{\geq 0}$. Let $\Gamma_1,\ldots,\Gamma_{\ell}<\Gamma$ be finite-index subgroups, $w_2,\ldots,w_{\ell}\in \mathsf{SL}_{pd+r}(\mathbb{K})$, $\zeta>0$, and consider the representations $\rho_{p,r},\psi_{p,r}$ and $\phi_{p,r}(\cdot,g_1,\ldots,\\g_{\ell}):\Gamma_1\ast \cdots \ast  \Gamma_{\ell}\rightarrow \mathsf{SL}_{pd+r}(\mathbb{K})$, $g_1,\ldots,g_{\ell}\in \mathsf{GL}_{pd+r}(\mathbb{K})$, $||g_i-I_{pd+r}||< \zeta$, from Theorem \ref{stable1}. If $2 \leq p\leq d-1$, the representation $\phi_{r,p}(\cdot,g_1,\ldots,g_{\ell})$  is non-locally rigid.\end{lemma}

\begin{proof} Set $q:=pd+r$ and consider the decomposition $\mathbb{K}^{q}=V_0\oplus V_1\oplus \cdots \oplus V_{p},$ $V_0=\mathbb{K}^r, V_i=\mathbb{K}^d$, $i>0$, for which:\\
\noindent \textup{(i)} $\rho_{p,r}(\Gamma_1)$ acts trivially on $V_0$, preserves and restricts to $\Gamma_1$ on $V_i$,\\
\noindent \textup{(ii)} $\psi_{p,r}(\Gamma_2)$ acts trivially on $V_0, V_1,\ldots, V_{p-1}$ and restricts to $\Gamma_2$ on $V_p$. 

 Since $\Gamma$ is $1$-Anosov and irreducible, $\Gamma_1,\Gamma_2<\mathsf{SL}_d(\mathbb{K})$ are also irreducible. In particular, the centralizer $\mathcal{Z}(\rho_{p,r})$ of $\rho_{p,r}(\Gamma_1)$ is the group $\mathsf{GL}(V_0)\times \mathsf{H}$, where $\mathsf{H}<\mathsf{GL}(V_1\oplus \cdots \oplus V_p)$ is the subgroup of invertible block matrices of the form $(z_{ij}I_d)_{i,j=1}^p$, $z_{ij}\in \mathbb{K}$. The centralizer $\mathcal{Z}(\psi_{p,r})<\mathsf{GL}_q(\mathbb{K})$ of $\psi_{p,r}(\Gamma_2)$ is $\mathsf{GL}(V_0\oplus V_1\oplus \cdots \oplus V_{p-1})\times \mathbb{K}^{\ast}I_d$. 

Let us set $$\mathsf{S}_{g_1,g_2}:=\big(g_1\mathcal{Z}(\rho_{p,r})g_1^{-1}g_2w\big)\big(\mathcal{Z}(\psi_{p,r})(g_2w)^{-1}\big).$$ By the definition of $\mathsf{S}_{g_1,g_2}\subset \mathsf{GL}_{pd+r}(\mathbb{K})$, if $g\in \mathsf{GL}_{q}(\mathbb{K})\smallsetminus \mathsf{S}_{g_1,g_2}$, the representation $\phi_{p,r}(\cdot,g_1,gg_2,\ldots, g_{\ell})$ is not conjugate to $\phi_{p,r}(\cdot,g_1,g_2,\ldots, g_{\ell})$. Thus, in order to complete the proof of the lemma it suffices to check that $\mathsf{S}_{g_1,g_2}$ does not contain any open neighbourhood of $I_{q}$. To see this, fix $\omega \in V_p$ and write $(g_1^{-1}g_2w)\omega=v_0+v_1+\cdots+v_p, \ v_i\in V_i.$ Note that $g_2w\mathcal{Z}(\psi_{p,r})(g_2w)^{-1}$ fixes the line $[g_2w\omega]$, hence $$\textup{span}(\mathsf{S}_{g_1,g_2}\big(g_2w\omega)\big)=g_1\textup{span}\big( \mathcal{Z}(\rho_{p,r})(g_1^{-1}g_2w\omega)\big).$$ It is clear from the description of $\mathcal{Z}(\rho_{p,r})$ that $\textup{span}(\mathcal{Z}(\rho_{p,r})((g_1^{-1}g_2w)v_0)) $ is a subspace of the direct sum $V_0\oplus V\oplus \cdots \oplus V$ of $V_0$ with $p$-copies of $V:=\textup{span}(v_1,\ldots, v_p)\subset \mathbb{K}^d$. This shows $\textup{dim}_{\mathbb{K}}(\textup{span}(\mathsf{S}_{g_1,g_2}(g_2v_0)))\leq p^2+r<pd+r$ and the lemma follows. \end{proof}

\section{Proof of Theorem \ref{mainthm}}\label{proof-mainthm}

Margulis' superrigidity theorem \cite{Mar} for irreducible higher rank lattices has been established by Corlette \cite{Corlette} and Gromov--Schoen \cite{Gromov-Schoen} for rank one lattices in $\mathsf{Sp}(n,1)$, $n\geq 2$, or $\mathsf{F}_4^{(-20)}$ in the Archimedean and non-Archimedean setting respectively. For such a rank one lattice $\Delta$, superrigidity implies that every linear representation of $\Delta$ essentially extends to a representation of the ambient rank one Lie group and in particular is semisimple (i.e. a direct sum of irreducible representations). In particular, Weil's criterion \cite{Weil} for local rigidity, implies that every complex representation of $\Delta$ is locally rigid.

\begin{proof}[Proof of Theorem \ref{mainthm}] Let $\Gamma$ be a uniform lattice in $\mathsf{Sp}(n,1)$, $n \geq 2$. Fix an irreducible $1$-proximal representation $\rho_0:\mathsf{Sp}(n,1)\rightarrow \mathsf{SL}_{d}(\mathbb{R})$, of minimal dimension $d:=d(n)$, and an integer $r\in \mathbb{Z}_{\geq 0}$. The restriction \hbox{$\rho_0|_{\Gamma}$} is strongly irreducible, $1$-Anosov and by \cite{DGK, Zimmer}, $\rho_0(\Gamma)$ preserves and acts convex cocompactly on a properly convex domain $D\subset \mathbb{P}(\mathbb{R}^d)$ with $\mathcal{C}^1$-boundary. The Zariski closure of $\rho_0$ is not locally isomorphic to either $\mathsf{SL}_d(\mathbb{R})$ or $\mathsf{SO}(d-1,1)$, hence, by \cite{Benoist-closure}, $\rho_0(\Gamma)$ cannot act cocompactly on $D$. In particular, there is a flag $y\in \mathcal{F}_{1,d-1}(\mathbb{K})$ which is transverse to the limit set of $\rho_0(\Gamma)$ in $\mathcal{F}_{1,d-1}(\mathbb{K})\subset \mathbb{P}(\mathbb{K}^d)\times \mathsf{Gr}_{d-1}(\mathbb{K}^d)$, where $\mathbb{K}$ is either $\mathbb{R}$ or $\mathbb{C}$. 

\par Let also $1\leq k_1,\ldots,k_s \leq \frac{d}{2}$ be all possible integers such that $\rho_0|_{\Gamma}$ is $\{k_1,\ldots,k_s\}$-Anosov and choose an integer $2\leq p\leq d-1$ such that $pk_i\neq k_j$ for every $i,j$ (e.g. take $p>\max\{k_1,\ldots,k_s\}$). By applying Theorem \ref{stable1} for the $1$-Anosov subgroup $\rho_0(\Gamma)<\mathsf{SL}_{d}(\mathbb{K})$ and $\ell:=2(pd+r)^2-1$, we can find finite-index subgroups $\Gamma_1,\Gamma_2,\ldots, \Gamma_{\ell}<\Gamma$, $\zeta=\zeta(\rho_0)>0$ and $w_2,\ldots,w_{\ell}\in \mathsf{GL}_{q}(\mathbb{K})$, $q:=pd+r$, with the property: for $\rho_{p,r},\psi_{p,r}$ defined as in (\ref{rep-freeproduct}) and $g_1,\ldots,g_{\ell}\in \mathsf{GL}_{q}(\mathbb{K})$ with $||g_i-I_{q}||< \zeta$, the unique representation $\phi_{p,r}(\cdot,g_1,\ldots,g_{\ell}):\Gamma_1\ast \cdots \ast \Gamma_{\ell}\rightarrow \mathsf{SL}_{q}(\mathbb{K})$ satisfying $$\phi_{p,r}(\gamma,g_1,\ldots,g_{\ell})=\left\{\begin{matrix}
 g_1\rho_{p,r}(\gamma)g_1^{-1}, & \gamma \in \Gamma_1  \\ 
(g_iw_i)\psi_{p,r}(\gamma)(g_iw_i)^{-1}, &  \gamma \in \Gamma_i,\ i=2,\ldots, \ell 
\end{matrix}\right.$$ is a quasi-isometric embedding.
 
By the superrigidity of $\Gamma_1,\ldots,\Gamma_{\ell}<\mathsf{Sp}(n,1)$, $\rho_{p,r}:\Gamma_1\rightarrow \mathsf{SL}_{q}(\mathbb{K})$ and $\psi_{p,r}:\Gamma_i\rightarrow \mathsf{SL}_{q}(\mathbb{K})$, $i>1$, are locally rigid. This fact, combined with Lemma \ref{embedding-1} applied for $\rho_{p,r}$ and $\psi_{p,r}$, shows that there exists an open neighbourhood $U_{1}$ (resp. $U_{i}$) of $\rho_{p,r}$ (resp. $w_i\psi_{p,r}w_i^{-1})$) such that every representation in $U_{1}$ (resp. $U_{i}$) is of the form $g\rho_{p,r}g^{-1}$ (resp. $gw_i\psi_{p,r}w_i^{-1}g^{-1}$) for some $g$ with $||g-I_{q}||< \zeta$. Now let $$\Omega_0:= U_1\times U_2\times \cdots \times U_{\ell}$$ be the open neighbourhood of $\phi_{p,r}(\cdot,I_q,\ldots,I_q)$ in $\textup{Hom}(\Gamma_1\ast\cdots \ast \Gamma_{\ell},\mathsf{SL}_q(\mathbb{K}))$. We verify that $\Omega_0$ has the following properties.
\medskip

\noindent \textup{(i)} First, by Theorem \ref{stable1}, the choice of $U_1,U_2,\ldots, U_{\ell}$ and the previous remarks, we check that every representation in the open set $\Omega_0$ is a quasi-isometric embedding.\\
 \noindent \textup{(ii)} $\Omega_0$ does not contain Anosov representations of $\Gamma_1\ast\cdots \ast \Gamma_{\ell}$ in $\mathsf{SL}_{q}(\mathbb{K})$. To see this, observe from Definition \ref{Anosov-def} that $\rho_{p,r}$ is $\{pk_1,\ldots,pk_s\}$-Anosov and $\psi_{p,r}$ is $\{k_1,\ldots,k_s\}$\footnote{Note that if $g\in \mathsf{Sp}(n,1)$ is hyperbolic and $\chi:\mathsf{Sp}(n,1)\rightarrow \mathsf{GL}_d(\mathbb{R})$ is a representation where $\chi(g)$ is $\{a_1,\ldots,a_s\}$-proximal, by adding diagonally an identity block to $\chi(g)$, the resulting matrix remains $\{a_1,\ldots,a_s\}$-proximal.}-Anosov. Now, if $\phi\in \Omega_0$ is $j$-Anosov for some $1\leq j \leq \frac{q}{2}$, both restrictions $\phi|_{\Gamma_1}$ and $\phi|_{\Gamma_2}$ have to be $j$-Anosov. However, this contradicts our choice of $p\in \mathbb{N}$ since $pk_i\neq k_j$ for every $i,j$. 
\medskip

\noindent \textup{(iii)} Since $p<d$, Lemma \ref{non-rigid} ensures that every representation in $\Omega_0$ is non-locally rigid. Now fix $w_0\in \Gamma_2\cap \cdots \cap \Gamma_{\ell}$ an infinite order element in $\Gamma$. Since $\ell=2q^2-1$, by applying Lemma \ref{Z-dense} for the $\{1,q-1\}$-proximal matrix $\psi_{p,r}(w_0)\in \mathsf{SL}_q(\mathbb{K})$, we can choose $f_2,\ldots,f_{\ell}\in \mathsf{SL}_{q}(\mathbb{K})$, arbitrarily close to $I_q$, such that $\phi_{p,r}(\cdot, I_q,f_2,\ldots, f_{\ell})\in \Omega_0$ and its restriction on $\Gamma_2\ast \cdots \ast \Gamma_{\ell}$ is Zariski dense in $\mathsf{SL}_q(\mathbb{K})$.  Then, by \cite[Prop. 8.2]{A'Campo-Burger} the set of Zariski dense representations of $\Gamma_{1}\ast\cdots \ast \Gamma_{\ell}$ in $\mathsf{SL}_n(\mathbb{K})$ is an open dense subset of the component containing $\phi_{p,r}(\cdot,I_q,\ldots,I_q)$. Thus, there exists an open subset $\Omega\subset \Omega_0$ containing \hbox{entirely of Zariski dense representations and satisfying (i) and (ii).}

Thus $\Omega$ has the properties claimed in the theorem and this concludes the proof. \end{proof}

Theorem \ref{mainthm} provides examples of robust quasi-isometric embeddings of hyperbolic groups, of arbitrarily large cohomological dimension, which are not limits of Anosov representations in either $\mathsf{SL}_m(\mathbb{R})$ or $\mathsf{SL}_m(\mathbb{C})$. By applying the method of proof of Theorem \ref{mainthm} for certain low dimensional representations of lattices in $\mathsf{Sp}(2,1)$ we establish examples in $\mathsf{SL}_m(\mathbb{C})$, $m\geq 30$.

\begin{corollary}\label{cor-1} Let $m\geq 30$ be an integer. There exist uniform lattices $\Delta_1,\ldots, \Delta_{\ell}<\mathsf{Sp}(2,1)$ and a Zariski dense, non-locally rigid, robust quasi-isometric embedding $\rho:\Delta_1\ast \cdots \ast \Delta_{\ell} \rightarrow \mathsf{SL}_m(\mathbb{C})$ which is not a limit of Anosov representations of $\Delta_1\ast \cdots \ast \Delta_{\ell}$ in $\mathsf{SL}_m(\mathbb{C})$.\end{corollary}

\begin{proof} Let $\tau:\mathsf{GL}_3(\mathbb{H})\rightarrow \mathsf{GL}_6(\mathbb{C})$ be the embedding obtained by realizing the quaternions $\mathbb{H}$ in $\mathsf{Mat}_{2\times 2}(\mathbb{C})$. Fix $\Delta<\mathsf{Sp}(2,1)<\mathsf{GL}_3(\mathbb{H})$ a uniform lattice and consider $\psi:\Delta \rightarrow \mathsf{GL}(\wedge^2\mathbb{C}^6)$ the composition of $\tau|_{\Delta}$ with the second exterior power \hbox{$\wedge^2:\mathsf{GL}_6(\mathbb{C})\rightarrow \mathsf{GL}(\wedge^2\mathbb{C}^6)$.} Any hyperbolic isometry in $\mathsf{Sp}(2,1)$, conjugate to the diagonal matrix $\textup{diag}(e^b,1,e^{-b})$, $b>0$, is mapped via $\wedge^2\tau$ to a matrix conjugate to $\textup{diag}(e^{2b}, e^b I_4,I_5,e^{-b}I_4,e^{-2b})$. In particular, $\psi:\Delta \rightarrow \mathsf{GL}(\wedge^2\mathbb{C}^6)$ is $\{1,5\}$-Anosov. The limit set of $\Delta<\mathsf{Sp}(2,1)$ in $\mathbb{P}(\mathbb{H}^3)$ is the $7$-sphere $$\big\{[(a_1+a_3j, a_2+a_4j,1)]:|a_1|^2+\cdots+|a_4|^2=1\big\},$$ \hbox{the limit sets of $\psi(\Delta)=\wedge^2\tau(\Delta)$ in $\mathbb{P}(\wedge^2\mathbb{C}^6)$ and $\mathsf{Gr}_{14}(\wedge^2\mathbb{C}^6)$ respectively are} \begin{align*}\xi_{\psi}^1(\partial_{\infty}\Delta)&=\big\{[(a_1,a_2,1,\overline{a_3},\overline{a_4},0) \wedge (-a_3,-a_4,0,\overline{a_1},\overline{a_2},1)]:|a_1|^2+\cdots+|a_4|^2=1\big\}\\ \xi_{\psi}^{14}(\partial_{\infty}\Delta)&=\big\{ v^{\perp}: [v]\in \xi_{\psi}^{1}(\partial_{\infty}\Delta)\big\},\end{align*} where $v^{\perp}$ is the orthogonal complement of $v\in \wedge^2\mathbb{C}^{6}$ with respect to the inner product on $\wedge^2\mathbb{C}^6$, for which $\{e_i\wedge e_j:i<j\}$ is an orthonormal basis. Note that $\big( [e_1\wedge e_4+e_2\wedge e_5], (e_3 \wedge e_6)^{\perp}\big)$ is transverse to the limit set $\{(\xi_{\psi}^1(\eta),\xi_{\psi}^{14}(\eta)):\eta\in \partial_{\infty}\Delta\}$ of $\psi$ in $\mathcal{F}_{1,14}(\mathbb{C})$. In addition, we check that the subspace spanned by $\xi_{\psi}^1(\partial_{\infty}\Delta)$ is the hyperplane $V=\big(e_1\wedge e_4+ e_2\wedge e_5-e_3 \wedge e_6)^{\perp}$. Since $\Delta$ has finite abelianization, up to passing to a finite-index subgroup of $\Delta$, we may assume $$\psi(\gamma):=\begin{pmatrix} 1 & \\ & \psi_1(\gamma)\end{pmatrix} \ \ \gamma \in \Delta,$$ where $\psi_1:\Delta\rightarrow \mathsf{SL}_{14}(\mathbb{C})$ is the strongly irreducible restriction \hbox{of $\psi$ on $V$.} 

Let $r\in \mathbb{Z}_{\geq 0}$. There is a flag $y\in \mathcal{F}_{1,14}(\mathbb{C})$ transverse to the limit set of $\psi$ and note that for every finite-index subgroup $\Delta'<\Delta$, Lemma \ref{embedding-1} applies to the nearby deformations of the locally rigid representations $\gamma \mapsto \textup{diag}(I_r, \psi(\gamma), \psi(\gamma))$ and $\gamma \mapsto \textup{diag}(I_r, I_{15},\psi(\gamma))$ of $\Delta'$ in $\mathsf{SL}_{30+r}(\mathbb{C})$. Since $\psi$ is $\{1,5\}$-Anosov, by using Theorem \ref{stable1} and applying the argument of the proof of Theorem \ref{mainthm} for $p=2$, $r\in \mathbb{Z}_{\geq 0}$ and $\ell:=2(30+r)^2-1$, there exist torsion-free finite-index subgroups $\Delta_1,\Delta_2\ldots, \Delta_{\ell}<\Delta$, $w_2,\ldots,w_{\ell}\in \mathsf{SL}_{30+r}(\mathbb{C})$ and a Zariski dense representation $\rho:\Delta_1\ast \cdots \ast \Delta_{\ell}\rightarrow \mathsf{SL}_{30+r}(\mathbb{C})$, where $$\rho|_{\Delta_1}(\gamma)=\textup{diag}(I_r,\psi(\gamma),\psi(\gamma)),\ \rho|_{\Delta_i}(\delta)=w_i\textup{diag}(I_r,I_{16},\psi_1(\delta))w_i^{-1}$$ for $i>1$, which is a stable quasi-isometric embedding and not a limit of Anosov representations in $\mathsf{SL}_{30+r}(\mathbb{C})$.\end{proof}

\begin{rmks} \normalfont{We assume the notation of the proof of Theorem \ref{mainthm}.\\
\noindent \textup{(i)} For every integer $k=1,\ldots,p+\left \lfloor \frac{r}{d} \right \rfloor$, the free product $\Gamma_1\ast \cdots \ast \Gamma_{\ell}$ admits $k$-Anosov representations in $\mathsf{SL}_{pd+r}(\mathbb{K})$. Indeed, since $\rho_0(\Gamma)$ acts convex cocompactly and not cocompactly on a properly convex domain in $\mathbb{P}(\mathbb{R}^d)$, by \cite{DGK2} (see \cite[Prop. 12.4]{DGK}), the group $\Gamma \ast \mathbb{Z}$ (and hence $\Gamma_1\ast \cdots \ast \Gamma_{\ell}$) admits a $1$-Anosov representation in $\mathsf{SL}_{d}(\mathbb{R})$. By using block diagonal embeddings of $\mathsf{SL}_d(\mathbb{R})$ in $\mathsf{SL}_{pd+r}(\mathbb{K})$ it is easy to check that $\textup{Anosov}_{k}(\Gamma_1\ast \cdots \ast \Gamma_{\ell},\mathsf{SL}_{pd+r}(\mathbb{K}))$ is not empty when $k=1,\ldots,p+\left \lfloor \frac{r}{d} \right \rfloor$.
\medskip

\noindent \textup{(ii)} Fix a word metric $d_{\mathsf{H}}$ on the free product $\mathsf{H}:=\Gamma_1\ast \cdots \ast \Gamma_{\ell}$. The proof of Theorem \ref{stable1} shows that there are uniform contants $K,C>0$ such that the orbit map of any representation $\phi\in \Omega$ is a $(K,C)$-quasi-isometric embedding of $(\mathsf{H},d_{\mathsf{H}})$ in the symmetric space $\mathsf{SL}_{pd+r}(\mathbb{K})/\mathsf{K}_{pd+r}$ (equipped with the Killing metric). This is also a property of stable neighbourhoods of Anosov representations (e.g. see \cite[Thm. 5.14]{GW}).} \end{rmks}

\begin{appendix}\section{Nearby conjugates of some Anosov representations}
In this appendix, we prove a technical lemma for nearby conjugates of the diagonal embeddings $\rho_{p,r}$ and $\psi_{p,r}$ defined in Theorem \ref{stable1}. We also prove a lemma that we use to verify that the open set $\Omega$ in Theorem \ref{mainthm} contains Zariski dense representations.

\begin{lemma}\label{embedding-1} Let $\mathbb{K}=\mathbb{R}$ or $\mathbb{C}$ and $\Delta<\mathsf{SL}_{d}(\mathbb{K})$ an irreducible $1$-Anosov subgroup. For $p\in \mathbb{N}$, $r\in \mathbb{Z}_{\geq 0}$, let $\rho_{p,r}:\Delta \rightarrow \mathsf{SL}_{pd+r}(\mathbb{K})$, $\rho_{p}:\Delta \rightarrow \mathsf{SL}_{pd}(\mathbb{K})$ be the representations defined as follows: $$ \rho_{p,r}(\gamma):=
\begin{pmatrix}
I_r & \\ 
 & \rho_p(\gamma)
\end{pmatrix}, \ \rho_p(\gamma):=\textup{diag}\big(\underbrace{\gamma,\ldots, \gamma}_{\textup{p-times}}\big) \ \ \gamma \in \Delta.$$ Let $\epsilon>0$. There exists an open neighbourhood $U_{\epsilon}\subset \textup{Hom}(\Delta,\mathsf{SL}_{pd+r}(\mathbb{K}))$ of $\rho_{p,r}$ with the property: if $\chi \in U_{\epsilon}$ is conjugate to $\rho_{p,r}$, there is $g\in \mathsf{GL}_{pd+r}(\mathbb{K})$ such that $||g-I_{pd+r}||<\epsilon$ and $$\chi(\gamma)=g\rho_{p,r}(\gamma)g^{-1}, \ \gamma \in \Delta.$$\end{lemma} 

\begin{proof} We first prove the lemma for $r=0$ and the representation $\rho_p:=\rho_{p,0}$.

Let $\xi^1:\partial_{\infty}\Delta\rightarrow \mathbb{P}(\mathbb{K}^d)$ be the Anosov limit map of $\Delta<\mathsf{SL}_{d}(\mathbb{K})$. Since $\Delta$ is irreducible, the centralizer $\mathcal{Z}(\rho_{p})<\mathsf{GL}_{pd}(\mathbb{K})$ of $\rho_{p}(\Delta)$ is the group  $\mathsf{GL}_{pd}(\mathbb{K})\cap \{(\lambda_{ij}I_{d})_{i,j=1}^{p}: \lambda_{ij}\in \mathbb{K}\}$, where $(\lambda_{ij}I_{d})_{i,j=1}^{p}$ is the $pd\times pd$ block matrix whose $(i,j)$-block is $\lambda_{ij}I_d$. In order to prove the lemma, it suffices to prove:
\vspace{0.1cm}

\noindent {\em Claim 1.} {\em For every sequence $(g_n)_{n\in \mathbb{N}}\subset \mathsf{GL}_{pd}(\mathbb{K})$ with $\lim_n g_n \rho_p g_n^{-1}=\rho_p$, $\lim_{n} g_n \mathcal{Z}(\rho_p)=\mathcal{Z}(\rho_p).$}
\medskip

\noindent {\em Proof of Claim 1.} We use induction on $p\in \mathbb{N}$. Suppose $p=1$ and observe that $\lim_{n}g_n \xi^1=\xi^1$. Up to passing to a subsequence, we may assume $\lim_n \frac{g_n}{||g_n||}=g_{\infty}$ and fix $U\subset \partial_{\infty}\Delta$ open such that $g_{\infty}v_{\eta}\neq 0$, for $\eta \in U$ and $\xi^1(\eta)=[v_{\eta}]$ . In particular, $g_{\infty}$ is invertible and $g_{\infty}\xi^1(\eta)=\xi^1(\eta)$ for every $\eta\in \partial_{\infty}\Delta$. This forces $g_{\infty}$ to be a unit scalar multiple of $I_d$. The claim follows for $p=1$.

Suppose that $p\geq 2$ and the claim holds for $p-1$. By Definition \ref{Anosov-def}, $\rho_p$ is $p$-Anosov with $p$-limit map $\xi^p:\partial_{\infty}\Delta\rightarrow \mathsf{Gr}_p(\mathbb{K}^{pd})$, \begin{align}\label{p-map}\xi^p(\eta):=\textup{span}\big((v_{\eta},0,\ldots,0)^t,(0,v_{\eta},\ldots,0)^t,\ldots, (0,\ldots,0,v_{\eta})^t\big)\end{align}  where $\xi^{1}(\eta)=[v_{\eta}],  \eta\in \partial_{\infty}\Delta$. In addition, observe that since $\Delta$ is irreducible, every $d$-dimensional $\mathbb{K}$-subspace of $\mathbb{K}^{pd}$, invariant under $\rho_p(\Delta)$, is of the form $$V_{r_1,\ldots, r_p}:=\big\{(r_1v,\ldots, r_pv)^t:v\in \mathbb{K}^{d}\big\}, \ r_1,\ldots, r_p \in \mathbb{K}.$$ Up to passing to a subsequence, we may assume that $V:=\lim_n g_nV_{1,0,\ldots,0}$ exists. Since $\lim_n g_n\rho_pg_n^{-1}=\rho_p$, $V$ is $\rho_p(\Delta)$-invariant and, by the definition of $\mathcal{Z}(\rho_p)$, choose $g_0\in \mathcal{Z}(\rho_p)$ such that $g_0V=V_{1,0,\ldots,0}$. Let $g_n':=g_{0}g_n$ and note $\lim_n g_n'\rho_p (g_n')^{-1}=\rho_p$, $\lim_n g_n'\xi^p=\xi^p$. Writing $\xi^1(\eta)=[v_{\eta}]$, since $\xi^p(\eta)\cap V_{1,0,\ldots,0}=\textup{span}( (v_{\eta},0,\ldots,0)^t)$, we check that $\lim_n [g_n'(v_{\eta},0\ldots,0)^t]$ exists and \begin{align}\label{convergence-1} \lim_{n\rightarrow \infty} g_n'[(v_{\eta},0,\ldots 0)^t]=[(v_{\eta},0,\ldots,0)^t], \ \eta\in \partial_{\infty}\Delta.\end{align} Now let us write $$g_n':=\begin{pmatrix}
A_n & B_n \\ 
C_n & D_n
\end{pmatrix}w_n$$ for some $w_n\in \mathcal{Z}(\rho_p), A_n\in \mathsf{Mat}_{d\times d}(\mathbb{K}), ||A_n||=1$ and $D_n\in \mathsf{Mat}_{s\times s}(\mathbb{K})$, $s:=(p-1)d$. For $\eta\in \partial_{\infty}\Delta$,  (\ref{convergence-1}) implies that \begin{align} \label{convergence-2} \lim_{n\rightarrow \infty}[A_nv_{\eta}]=[v_{\eta}], \ \lim_{n\rightarrow \infty}\frac{||C_nv_{\eta}||}{||(A_nv_{\eta},C_nv_{\eta})^t||}=0.\end{align} In particular, by using the same argument as in the case $p=1$, we obtain a sequence $(\lambda_n)_{n}$, $|\lambda_n|=1$, with $\lim_n \lambda_n A_n=I_d$. Then, (\ref{convergence-2}) implies $\lim_nC_nv_{\eta}=0$ for every $\eta$ and $\lim_nC_n=0_{s\times d}$. Thus, since $\lambda_nA_n$ is eventually invertible, we may write $$g_n':=\begin{pmatrix}
\lambda_n A_n &  \\ 
\lambda_n C_n & I_{s}
\end{pmatrix}g_n''w_n', \ g_n'':=\begin{pmatrix}
I_d & F_n \\ 
 & D_n
\end{pmatrix},$$ for some $w_n'\in \mathcal{Z}(\rho_p)$. It is clear that since $w_n'\in \mathcal{Z}(\rho_p)$ we have $\lim_n g_n''\rho_p(g_n'')^{-1}\\=\rho_p$. In particular, $\lim_n D_n \rho_{p-1} D_{n}^{-1}=\rho_{p-1}$ and by the inductive step we may choose $(c_n)_{n\in \mathbb{N}}\subset \mathcal{Z}(\rho_{p-1})$ such that $\lim_{n}D_nc_n=I_{s}, s=(p-1)d$.  Thus we may write: $$g_n':=\begin{pmatrix}
\lambda_n A_n &  \\ 
\lambda_n C_n & I_{s}
\end{pmatrix}g_n'''w_n'', \ g_n''':=\begin{pmatrix}
I_d & F_n \\ 
 & I_{s}
\end{pmatrix}, \  w_n''= \begin{pmatrix}
I_d & \\ 
 & c_n
\end{pmatrix}w_n'\in \mathcal{Z}(\rho_p)$$  where $F_n\in \mathsf{Mat}_{d\times s}(\mathbb{K})$ is in block form $$F_n:=(F_{n,1} \ldots, F_{n,p-1}),\  F_{n,i}\in \mathsf{Mat}_{d\times d}(\mathbb{K}).$$ Since $w_n''\in \mathcal{Z}(\rho_p)$, $\lim_{n} g_n'''\rho_p (g_n''')^{-1}=\rho_p$ and $\lim_{n} (F_{n,i}\gamma-\gamma F_{n,i})=0_{d\times d}$ for every $\gamma \in \Delta$ and $i=1,\ldots,p-1$. Observe that since $\Delta$ is irreducible and contains a $1$-proximal element, by Burnside's theorem (see the proof in \cite{Burnsidethm}), $\Delta$ linearly spans $\mathsf{Mat}_{d\times d}(\mathbb{K})$, hence \begin{align}\label{convergence-3} \lim_{n\rightarrow \infty} (F_{n,i}X-XF_{n,i})=0_{d\times d} \ \ \forall X\in \mathsf{Mat}_{d\times d}(\mathbb{K}), \ i=1,\ldots, p.\end{align} By applying (\ref{convergence-3}) for the elementary matrix $X:=\mathcal{E}_{mj}$ (where $\mathcal{E}_{mj}$ is the $d\times d$ matrix with 1 in the $(m,j)$ entry and $0$ everywhere else), we obtain $$\lim_{n \rightarrow \infty} (a_{11}(F_{n,i})-a_{jj}(F_{n,i}))=0, \  \lim_{n\rightarrow \infty} a_{mj}(F_{n,i})=0$$ for $m\neq j, i=1,\ldots, p-1$ and $a_{mj}(F_{n,i})$ denotes the $(i,j)$-entry of $F_{n,i}$. This shows that for every $i$, we have $\lim_{n}(F_{n,i}-\mu_{i,n}I_{d})=0_{d\times d}$, where $\mu_{i,n}:=a_{11}(F_{n,i})$. Hence, $\lim_{n} g_{n}'''b_n=I_{pd}$, where $$b_{n}:=\begin{pmatrix}
I_d & T_n \\ 
 & I_{s}
\end{pmatrix}\in \mathcal{Z}(\rho_p), \ T_n:=\big(-\mu_{1,n}I_d, \ldots, -\mu_{p-1,n}I_d\big).$$ In conclusion, for large $n\in \mathbb{N}$, $$g_0g_n=g_n'=\begin{pmatrix}
\lambda_n A_n &  \\ 
\lambda_n C_n & I_{s}
\end{pmatrix}(g_n'''b_n)b_{n}^{-1}w_n'', \ \ w_n'',b_n\in \mathcal{Z}(\rho_p),$$  $\lim_n\lambda_n A_n=I_d$, $|\lambda_n|=1$, $\lim_n C_n=0_{s\times d}$, $\lim_n g_n'''b_n=I_d$ and hence $$\lim_{n\rightarrow \infty}g_0g_n(w_n'')^{-1}b_n=I_{d}.$$ Finally, since $g_0\in \mathcal{Z}(\rho_p)$, $\lim_n g_n\mathcal{Z}(\rho_p)=\mathcal{Z}(\rho_p)$. This finishes the proof of the induction and of Claim 1.
\medskip

\noindent {\em Proof of Lemma \ref{embedding-1} for $r\geq 1$.} Let $q:=pd+r$, and consider the decomposition $\mathbb{K}^q=W_1\oplus W_2$, $W_1=\mathbb{K}^r, W_2=\mathbb{K}^{d}\oplus \cdots \oplus \mathbb{K}^d$, with respect to which $\rho_{p,r}$ preserves and acts trivially on $W_1$ and restricts to $\rho_p$ on $W_2$. By the irreducubility of $\Delta$, the centralizer of $\mathcal{Z}(\rho_{p,r})<\mathsf{GL}(W_1\oplus W_2)$ of $\rho_{p,r}(\Delta)$ is $\mathsf{GL}(W_1)\times \mathcal{Z}(\rho_{p})$. 

In order to complete the proof, it suffices to prove that for any sequence $(h_n)_{n\in \mathbb{N}}\subset \mathsf{GL}_{q}(\mathbb{K})$ with $\lim_{n}h_n\rho_{p,r}h_n^{-1}=\rho_{p,r}$, $\lim_n h_n \mathcal{Z}(\rho_{p,r})=\mathcal{Z}(\rho_{p,r})$. Note that $\rho_{p,r}$ is $p$-Anosov with $(q-p)$-limit map $\xi^{q-p}:\partial_{\infty}\Delta\rightarrow \mathsf{Gr}_{q-p}(\mathbb{K}^q)$, $\xi^{q-p}(\eta):=W_1 \oplus \xi^{d-1}(\eta)\oplus \cdots \oplus \xi^{d-1}(\eta)$, $\eta\in \partial_{\infty}\Delta$, where $\xi^{d-1}$ is the $(d-1)$-limit map of $\Delta$. Since $\lim_n h_n \xi^{q-p}=\xi^{q-p}$ and $\bigcap_{\eta}\xi^{q-p}(\eta)=W_1$, we have $\lim_nh_nW_1=W_1$. Writing $h_n$ in block-form, eventually, the top left $r\times r$-block of $h_n$ is invertible (otherwise, there would exist $(v_n)_{n\in \mathbb{N}}\subset W_1$ non-zero such that $h_nv_n\in W_2$ for every $n$). Up to right mutliplication by an element of $\mathcal{Z}(\rho_{p,r})$, we may write $$h_n:=\begin{pmatrix}
I_r & Z_n \\ 
 L_n& J_n
\end{pmatrix}, \ J_n\in \mathsf{Mat}_{pd\times pd}(\mathbb{K}).$$ Since $\lim_n h_nW_1=W_1$, this forces $\lim_n L_n=0_{pd\times r}$ and we may write: $$h_n\mathcal{Z}(\rho_{p,r})=\begin{pmatrix}
I_r &  \\ 
 L_n& I_{pd}
\end{pmatrix}\begin{pmatrix}
I_r & Z_n \\ 
 & J_n'
\end{pmatrix}\mathcal{Z}(\rho_{p,r}).$$ Notice that $\lim_n J_n'\rho_p(J_n')^{-1}=\rho_p$, thus, by Claim 1, there is a sequence $(\omega_n)_{n\in \mathbb{N}}\subset \mathcal{Z}(\rho_{p})$ such that $\lim_n J_n'\omega_n=I_{pd}$. In particular we may write: $$h_n\mathcal{Z}(\rho_{p,r})=\begin{pmatrix}
I_r &  \\ 
 L_n& J_n'\omega_n
\end{pmatrix}\begin{pmatrix}
I_r & Z_n '\\ 
& I_{pd}
\end{pmatrix}\mathcal{Z}(\rho_{p,r})$$ and $\lim_n Z_n'(\rho_{p}(\gamma)-I_{pd})=0_{r\times pd}.$ Writing $Z_n'$ in $(d\times d)$-blocks, if $Z_{n,1}',\ldots, Z_{n,pr}'\\ \in \mathsf{Mat}_{d\times d}(\mathbb{K})$ are its blocks, we see that $\lim_n Z_{n,i}'(\gamma-I_d)=0_{d\times d}$ for every $\gamma \in \Delta$. Thus $\lim_n Z_{n,i}'=0_{d\times d}$, since $\Delta$ linearly spans $\mathsf{Mat}_{d\times d}(\mathbb{K})$. 

Hence, for every sequence $(h_n)_{n\in \mathbb{N}}\subset \mathsf{GL}_{q}(\mathbb{K})$ with $\lim_n h_r \rho_{p,r}h_n^{-1}=\rho_{p,r}$ we have $\lim_n h_n\mathcal{Z}(\rho_{p,r})=\mathcal{Z}(\rho_{p,r})$. This completes the proof of the lemma.\end{proof}

\subsection{Zariski dense deformations} For a $\{1,q-1\}$-proximal matrix $w\in \mathsf{GL}_q(\mathbb{K})$ denote by $x_w^{\pm}$ (resp. $V_w^{\pm}$) the attracting fixed point of $w^{\pm 1}$ in $\mathbb{P}(\mathbb{K}^q)$ (resp. $\mathsf{Gr}_{q-1}(\mathbb{K}^q)$). We will use the following elementary lemma to exhibit Zariski dense examples in Theorem \ref{mainthm} .

\begin{lemma} \label{Z-dense} Fix $q\in \mathbb{N}_{\geq 2}$ and $(v_0^{\pm},V_0^{\pm})\in \mathcal{F}_{1,q-1}(\mathbb{K}^q)$ two transverse flags. Let $\mathcal{O}_1,\ldots, \mathcal{O}_{s}$, $s:=2q^2-2$, be open subsets of $\mathsf{GL}_{q}(\mathbb{K})$. There exist $f_1\in \mathcal{O}_1, \ldots, f_{s}\in \mathcal{O}_s$, with the property: for every $\{1,q-1\}$-proximal matrix $w\in \mathsf{GL}_q(\mathbb{K})$ with $x_w^{\pm}=v_0^{\pm},V_w^{\pm}=V_0^{\pm}$, $\langle f_1wf_1^{-1},\ldots, f_{s}wf_{s}^{-1}\rangle$ is a Zariski dense subgroup of $\mathsf{GL}_{q}(\mathbb{K})$.\end{lemma}

For the proof of the previous lemma we need the following observation.

\begin{lemma} \label{Z-dense2} Let $G$ be a strongly irreducible subgroup of $\mathsf{GL}_n(\mathbb{K})$. Fix $(\omega_0,V_0) \in \mathcal{F}_{1,n-1}(\mathbb{K}^n)$ and $\mathcal{O}_1,\ldots, \mathcal{O}_{2n}\subset G$ Zariski dense subsets. Then there exist $f_1\in \mathcal{O}_{1},\ldots, f_{2n}\in \mathcal{O}_{2n}$ with the property: for every $1$-proximal matrix $g \in G$ with $x_g^{+}=[\omega_0]$ and $V_g^{-}=V_0$, the group $\langle f_1 gf_1^{-1},\ldots, f_{2n}gf_{2n}^{-1}\rangle$ is irreducible.\end{lemma}

\begin{proof} We write $V_0=\{v\in \mathbb{K}^n:\langle v,\omega_1\rangle=0\}$, where $\langle \cdot,\cdot\rangle$ is the inner product on $\mathbb{K}^n$, for some $\omega_1\in \mathbb{K}^n$. Let $G^{2n}$ be the direct product of $2n$-copies of $G$. For a subset $I:=\{i_1<\cdots<i_n\}$ of $\{1,\ldots, 2n\}$, $|I|=n$, define the Zariski closed subsets of $\mathsf{GL}_{n}(\mathbb{K})\times \cdots \times \mathsf{GL}_{n}(\mathbb{K})$, \begin{align*}X_{I}(\omega_0)&:=\big\{(A_{1},\ldots,A_{2n}): \textup{det}\big(A_{i_1}\omega_0|\cdots | A_{i_n}\omega_0)=0\big\}\\ X_{I}(\omega_1)&:=\big\{(A_{1},\ldots,A_{2n}): \textup{det}(A_{i_1}^{-t}\overline{\omega_1} |\cdots | A_{i_n}^{-t}\overline{\omega_1})=0\big\}. \end{align*} Since $G$ and $G^{t}$ are strongly irreducible, for every finite-index subgroup $H<G$, $\textup{span}(H\omega_0)=\textup{span}(H^{t}\overline{\omega_1})=\mathbb{K}^n$, thus $H^{2n}$ is not contained in any of the sets $X_I(\omega_0)\cup X_I(\omega_1)$. In particular, $G^{2n}$ cannot be contained in $X:=\bigcup_{|I|=n}(X_I(\omega_0)\cup X_I(\omega_1))$. Since $\mathcal{O}_1\times \cdots \times \mathcal{O}_{2n}$ is Zariski dense in $G^{2n}$, choose $f_i\in \mathcal{O}_i$ with $(f_1,\ldots,f_{2n})\notin X$, meaning that $$\textup{span}(f_{i_1}\omega_1,\ldots, f_{i_n}\omega_1)=\mathbb{K}^n,\ f_{i_1}V_0\cap \cdots \cap f_{i_n}V_0=(0)$$ for every $n$-subset $I=\{i_1<\cdots<i_n\}$. Now suppose that $g\in G$ is a $1$-proximal element with $x_g^{+}=[\omega_0]$, $V_g^{-}=V_0$ and $W\subset \mathbb{K}^n$ a proper subspace invariant by $f_ig f_i^{-1}$ for every $i=1,\ldots, 2n$. If for some $i$, $W$ is not a subspace of $f_iV_0$, since $f_ig^nf_i^{-1}W=W$, we necessarily have $f_i\omega_0\in W$. Since $W$ is proper, there is a subset $J\subset \{1,\ldots, 2n\}$, $|J|\geq n+1$, with $W\subset f_jV_0$ for every $j\in J$ and hence $W\subset \bigcap_{j\in J}f_jV_0=(0)$. This shows that $ \langle f_1gf_1^{-1},\ldots, f_{2n}gf_{2n}^{-1} \rangle$ is irreducible.\end{proof}

\begin{proof}[Proof of Lemma \ref{Z-dense}] Let $\textup{Ad}:\mathsf{GL}_q(\mathbb{K})\rightarrow \mathsf{SL}^{\pm}_{q^2-1}(\mathbb{K})$ be the adjoint representation. If $w\in \mathsf{GL}_q(\mathbb{K})$ is a $\{1,q-1\}$-proximal element with $x_{w}^{\pm}=[v_0^{\pm}]$ and $V_{w}^{\pm}=V_0^{\pm}$, $\textup{Ad}(w)$ is $1$-proximal and its attracting fixed point and repelling hyperplane depend only on $v_0^{\pm},V_0^{\pm}$. Therefore, by Lemma \ref{Z-dense2}, there exist $f_1\in \mathcal{O}_i, \ldots, f_{s}\in \mathcal{O}_s$, $s=2q^2-2$, such that $\langle \textup{Ad}(w_1 ww_1^{-1}),\ldots, \textup{Ad}(w_{s}ww_{s}^{-1})\rangle$ is irreducible, or equivalently, that $\langle f_1 wf_1^{-1},\ldots, f_{s}wf_{s}^{-1}\rangle$ is Zariski dense in $\mathsf{SL}_q(\mathbb{K})$.\end{proof}

\end{appendix}

\bibliographystyle{siam}

\bibliography{biblio.bib}

\end{document}